\documentclass[10pt]{article}

\usepackage{amsmath}
\usepackage{amsthm}
\usepackage{amssymb}
\usepackage{enumitem}
\usepackage{multicol}
\usepackage{booktabs} 
\usepackage[active]{srcltx}
\usepackage{algorithm}
\usepackage{algpseudocode}
\usepackage{array}
\usepackage[dvipsnames,svgnames]{xcolor}
\usepackage{tikz}
\usetikzlibrary{positioning,arrows.meta}
\usepackage{graphicx}
\usepackage{subcaption}
\usepackage{pgfplots}
\usepackage{xcolor}
\usepackage{cancel}
\usepackage{hyperref}
\usepackage[top=35mm, bottom=25mm, left=30mm, right=30mm]{geometry}
\usepackage{comment}

\newtheorem{theorem}{Theorem}[section]

\newtheorem{proposition}[theorem]{Proposition}
\newtheorem{corollary}[theorem]{Corollary}
\newtheorem{assumption}[theorem]{Assumption}
\newtheorem{fact}[theorem]{Fact}

\newtheorem{remark}[theorem]{Remark}

\newcommand{\R}{\mathbb{R}}
\newcommand\Rinf{\overline{\mathbb{R}}}
\newcommand\Nz{\mathbb{N}_0}
\newcommand{\E}{\mathbb{E}}
\newcommand\ov[1]{\overline{#1}}
\newcommand\dom[1]{ \bs{{\rm dom}}(#1)} 
\newcommand\bs[1]{\boldsymbol{#1}}
\newcommand\prox[2]{\bs{{\rm Prox}}_{#2#1}}
\newcommand\fgam[2]{#1_{#2}}

\newcommand\fgamepsk[2]{#1_{#2}^{\varepsilon_k}}
\newcommand\dist{ \bs{{\rm dist}}} 
\newcommand{\Law}{\mathcal{L}}
\newcommand{\W}{\mathcal{W}_2}

\newcommand{\inner}[2]{\left\langle #1,#2\right\rangle}
\newcommand{\norm}[1]{\left\|#1\right\|}
\newcommand\U{\mathcal{U}}%
 \newcommand\Z{\mathcal{Z}}%
\DeclareMathOperator*{\argminold}{arg\,min}
\newcommand\argmin{\bs\argminold}%

\let\oldliminf\liminf
\renewcommand{\liminf}{\bs{\oldliminf}}
\let\oldlimsup\limsup
\renewcommand{\limsup}{\bs{\oldlimsup}}
\let\oldmax\max
\renewcommand{\max}{\bs{\oldmax}}
\let\oldmin\min
\renewcommand{\min}{\bs{\oldmin}}
\let\oldsup\sup
\renewcommand{\sup}{\bs{\oldsup}}
\let\oldinf\inf
\renewcommand{\inf}{\bs{\oldinf}}
\let\oldlim\lim
\renewcommand{\lim}{\bs{\oldlim}}

\title{{\bf Speeding Up Nonsmooth Bayesian MCMC Sampling via Inexact Proximal Unadjusted Langevin Algorithm}}
\author{
Susan Ghaderi \thanks{Department of Electrical Engineering (ESAT-STADIUS), KU Leuven, Belgium. Email: susan.ghaderi@kuleuven.be}
\and
Alireza Kabgani\thanks{Department of Mathematics, University of Antwerp, Antwerp, Belgium. Email: alireza.kabgani@uantwerp.be} 
\and
Yves Moreau\thanks{Department of Electrical Engineering (ESAT-STADIUS), KU Leuven, Belgium. Email: yves.moreau@kuleuven.be}
\and
Masoud Ahookhosh\thanks{Department of Mathematics, University of Antwerp, Antwerp, Belgium. Email: masoud.ahookhosh@uantwerp.be},
}

\begin{document}


\maketitle

\begin{abstract}
We study sampling from posterior distributions with nonsmooth composite potentials, a setting in which proximal-based Langevin methods are theoretically appealing but in practice limited to simple functions with closed-form proximal operators. We introduce iPULA for composite potentials, an inexact proximal unadjusted Langevin algorithm that replaces exact proximal steps with controlled approximations. Our approach leverages the Moreau envelope to smooth the potential, while allowing inexact evaluation of its gradient through inexact proximal computations. We establish non-asymptotic convergence guarantees for iPULA, explicitly characterizing the impact of inexactness on the sampling error and showing that the inexactness preserves convergence rates up to a quantifiable bias. We demonstrate the practical relevance of iPULA on a medical image reconstruction task, where proximal operators cannot be computed exactly. Experiments demonstrate the effectiveness of iPULA and support our theoretical results.
\noindent 

{\small}

\medskip



\end{abstract}

\section{Introduction}\label{sec:intro}
Markov chain Monte Carlo (MCMC) methods are a cornerstone of modern statistical inference, providing a practical framework for working with complex probabilistic models that are otherwise analytically intractable (see, e.g., \cite{robert2004monte}). By generating dependent samples from high-dimensional target distributions, MCMC enables the approximation of posterior distributions in Bayesian settings, even when normalizing constants are unknown. This flexibility has led to widespread adoption across diverse domains, including computational biology for phylogenetic inference, physics for simulating particle systems, machine learning for training deep generative models, and econometrics for hierarchical and time-series modeling. As data sizes continue to grow and models become increasingly complex, MCMC remains a fundamental bridge between statistical theory and practical inference, driving ongoing research into more efficient, scalable, and robust sampling algorithms.

While MCMC methods are asymptotically valid, their efficiency often degrades in high dimensions, where random-walk Metropolis exhibits slow, highly correlated, and poorly mixing behavior. As dimensionality increases, maintaining reasonable acceptance rates requires shrinking the proposal step-size, leading to inefficient traversal of the state space due to poor adaptation to the target distribution’s local structure and a high computational cost to adequately explore regions of high probability mass; see, e.g., \cite{brooks2011handbook}. This limitation motivates the use of gradient-based MCMC methods, such as those built on Langevin dynamics, which incorporate first-order derivative information to guide proposals toward more probable regions; see, e.g., \cite{dalalyan2017further,dalalyan2019user,dalalyan2022bounding,Durmus2019high,DurmusPereyra,ma2015complete}. By leveraging gradients, these methods adapt to the curvature of the target distribution, enabling larger, more informed moves that significantly improve convergence rates and sampling efficiency. Consequently, Langevin-based approaches offer a principled balance between exploration and exploitation, making them especially important for scalable inference in modern high-dimensional applications; see, e.g., \cite{dubey2016variance}.

\subsection{Related work: Nonsmooth MCMC methods and key question}\label{sec:motivation}

{\it Nonsmooth potentials} arise naturally in applications involving sparsity-inducing priors (e.g., $\ell_1$ penalties in compressed sensing and high-dimensional regression), image reconstruction with total variation regularization, nuclear-norm based low-rank matrix estimation, and Bayesian formulations of machine learning models with hinge or absolute-value losses (e.g., Bayesian trendfiltering); see, e.g., \cite{Crucinio2025optimal,DurmusPereyra,lau2022bregman,lau2022nonlog,Pereyra2016proximal,salim2019stochastic,ShuklaVatsChi2025}.
While gradient-based MCMC methods are not directly applicable in the nonsmooth setting, the need for efficient sampling in high-dimensional regimes remains a need. 
In this setting, unadjusted Langevin algorithm (ULA) and its variants can be extended either by replacing gradients with subgradients \cite{habring2024subgradient} or by introducing smooth surrogates of the potential. When coupled with the unadjusted Langevin algorithm (ULA), such smoothing strategies, including Gaussian smoothing \cite{chatterji2020langevin,nesterov2017randon}, the Moreau envelope \cite{DurmusPereyra,kabgani2025moreau,kabgani2024itsopt,kabgani2025itsdeal,klatzer2024accelerated,moreau1963pro,Pereyra2016proximal,PlanidenWang}, and forward–backward envelopes \cite{ahookhosh2021bregman,eftekhari2023forward,GhaderiSULA,stella2017forward,themelis2019acceleration}, yield well-defined stochastic dynamics that preserve key structural properties of the original objective while enabling efficient gradient-based exploration.

Owing to their computational efficiency, smoothing-based ULA methods provide a powerful and scalable framework for Bayesian inference in nonsmooth models, striking a balance between theoretical tractability and practical performance across a wide range of modern applications. The principal smoothing techniques integrated with ULA are as follows:
\begin{itemize}[leftmargin=*]
    \item {\bf Moreau envelope} is a smoothing technique that associates to a given function $\U$ a differentiable approximation obtained via infimal convolution with a scaled quadratic function, i.e.,
    \begin{equation}\label{eq:env_def}
        \fgam{\U}{\gamma}(x):=\inf_{y\in\R^n}\left\{\U(y)+\frac{1}{2\gamma}\norm{x-y}^2\right\},
    \end{equation}
    thereby preserving key properties such as convexity while it is smooth with the gradient
    \begin{equation}\label{eq:gradMoreau}
        \nabla \fgam{\U}{\gamma}(x)=\frac{1}{\gamma}\bigl(x-\prox{\U}{\gamma}(x)\bigr),
    \end{equation}
    where $\gamma>0$ and the {\it proximal operator} is given by
    \begin{equation}\label{eq:prox_def}
        \prox{\U}{\gamma}(x) :=\argmin_{y\in\R^n}\left\{\U(y)+\frac{1}{2\gamma}\norm{x-y}^2\right\}.
    \end{equation}
    \item {\bf Forward–backward envelope} (FBE) is a smoothing technique that offers an efficient way to generate a differentiable approximation of composite potentials of the form $\mathcal{U}(x)=f(x)+g(x)$, where $f$ is $L$-smooth with some $L>0$ and $g$ is simple with a closed-form proximal operator, i.e.,
    \begin{equation}\label{eq:fbe}
        \fgam{\U}{\gamma}^{f,g}(x):=\inf_{y\in\R^n}\left\{f(x)+\langle\nabla f(x),y-x\rangle+ g(y)+\frac{1}{2\gamma}\norm{x-y}^2\right\},
    \end{equation}
    which is smooth and its gradient is given by
    \begin{equation}\label{eq:gradFBE}
        \nabla \fgam{\U}{\gamma}^{f,g}(x)=\frac{1}{\gamma} \left(I-\gamma \nabla^2f(x)\right) \left(x-\prox{g}{\gamma}(x-\gamma\nabla f(x))\right),
    \end{equation}
    where $\gamma\in (0,\tfrac{1}{L}]$.
\end{itemize}

{\bf Limitations of existing smoothing techniques.} While the Moreau envelope and forward--backward envelope (FBE) preserve the minimizers of the original potential, they exhibit several limitations. For the {\it Moreau envelope}, the proximal operator admits a closed form only when the potential has a particularly simple structure, which is typically not satisfied in composite settings; as a result, its use within ULA is largely restricted to simple models. The applicability of {\it forward-backward envelope} is further constrained by more stringent requirements: (i) evaluating its gradient involves the Hessian of \(f\), which is computationally prohibitive in high dimensions; (ii) the assumption that \(f\) is \(L\)-smooth restricts the class of admissible models, excluding many problems with nonsmooth or relatively smooth data-fidelity terms; and (iii) it relies on the simplicity of \(g\) to ensure a tractable proximal operator, thereby excluding complex or nonseparable regularizers. Collectively, these limitations hinder the practicality of FBE-based ULA methods for many applications. Accordingly, a natural key question arises in the community:

``{\it How can we design fast and scalable unadjusted Langevin MCMC algorithms for nonsmooth composite potentials arising in Bayesian learning?}''

\subsection{Contribution}
To address this question, we first approximate the proximal operator in \eqref{eq:prox_def}, which in turn yields an inexact oracle for the Moreau envelope in \eqref{eq:env_def}. We then incorporate this inexact gradient into the unadjusted Langevin algorithm (ULA), resulting in an inexact proximal-based ULA scheme. The key elements of this construction are summarized below and form the main contributions of this paper:
\begin{itemize}[leftmargin=*]
    \item ({\bf Inexact oracle of Moreau envelope for ULA.}) In the composite setting, we first solve the proximal subproblem \eqref{eq:prox_def} approximately and substitute the result into \eqref{eq:gradMoreau}, thereby obtaining an inexact gradient of the Moreau envelope \eqref{eq:env_def}; see Section~\ref{sec:inexactOracle}. Notably, this inexact approach removes the need for $\U$ to admit a simple structure, in contrast to the exact Moreau and forward–backward envelope constructions discussed in Section~\ref{sec:motivation}. We then incorporate this inexact gradient into the ULA update, yielding the inexact proximal unadjusted Langevin algorithm (iPULA). To the best of our knowledge, this is the first proximal-based ULA method that leverages an inexact oracle for the Moreau envelope of composite potentials arising in Bayesian learning.
    \item ({\bf Non-asymptotic convergence analysis of iPULA.}) Assuming strong convexity of the original potential, the corresponding Moreau envelope is also strongly convex. We show that the sequence ${\E[\fgam{\U}{\gamma}(x_k) - \fgam{\U}{\gamma}(x^\star)]}{k \geq 0}$ converges linearly up to a residual constant, yielding a complexity of order $\mathcal{O}(\log(1/\varepsilon))$ to reach accuracy $\varepsilon > 0$; cf. Section~\ref{sec:iPULA}. We further establish that the sequence ${\mathcal{W}2(\Law(x_k), \pi\gamma)}{k \geq 0}$ exhibits linear convergence up to a constant, which likewise implies a complexity of order $\mathcal{O}(\log(1/\varepsilon))$; cf. Section~\ref{sec:iPULA}. Finally, numerical experiments demonstrate the promising performance of iPULA, corroborating our theoretical results; see Section~\ref{sec:experiments}.
\end{itemize}


\section{Preliminaries and notation}\label{sec:preliminaries}
Throughout the paper, 
$\R^n$ denotes  the $n$-dimensional \textit{Euclidean space} equipped
with \textit{inner product} $\inner{\cdot}{\cdot}$ and \textit{norm} $\norm{\cdot}$.  We write $\Rinf:=\R\cup\{+\infty\}$, and  denote the set of \textit{nonnegative integers} by $\Nz := \mathbb{N}\cup\{0\}$.  
For a function $h:\R^n\to\Rinf$, its \textit{effective domain} is
$\dom{h}:=\{x:h(x)<+\infty\}$.
The function $h$ is called \textit{proper} if $\dom h\neq\emptyset$.

A function $h: \mathbb{R}^{n} \rightarrow \overline{\mathbb{R}}$ with convex domain is said to be
\textit{strongly convex} on $\dom{h}$ if there exists $\sigma > 0$ such that, for all $x,y \in \dom{h}$ and all $\lambda \in [0,1]$,
 \begin{equation}\label{strong:convex}
  h(\lambda y + (1-\lambda)x) \leq \lambda h(y) + (1-\lambda) h(x) -
  \lambda (1 - \lambda) \frac{\sigma}{2} \lVert x - y \rVert^{2},
 \end{equation}

For a proper convex function $h:\R^n\to\Rinf$, the convex subdifferential of $h$ at
$x\in\R^n$ is defined by
\[
\partial h(x):=\left\{ v\in\R^n:h(y)\ge h(x)+\inner{v}{y-x}\quad \forall y\in\R^n\right\},
\]
with the convention $\partial h(x)=\emptyset$ if $x\notin\dom h$. Any vector $v\in\partial h(x)$ is called a subgradient of $h$ at $x$. The mapping
$\partial h$ is monotone, that is,
\[
    \inner{u-v}{x-y}\ge 0,  \qquad   \forall u\in\partial h(x),\; v\in\partial h(y).
\]
For any pair of probability measures $u$ and $v$ on the same parameter space, a transference plan $\xi$ between $u$ and $v$ is a joint distribution such that the marginal distributions on two sets of coordinates are $u$ and $v$, respectively. We denote by $\Gamma(u,v)$ the set of all transference plans, and define the $2$-Wasserstein distance between these two probability distributions as
\begin{align*}
    \W^2(u,v)=\inf_{\xi\in\Gamma(u,v)} \int \|x-y\|_2^2 d\xi(x,y).
\end{align*}

\section{Problem setup and Moreau envelope}\label{sec:prelem}
 Let us consider {\it sampling} from a canonical target distribution $\pi(x): \mathbb{R}^n\to \Rinf $ given by
\begin{equation}\label{eq:target}
    \pi(x):=\frac{e^{-\U(x)}}{\int_{\R^n} e^{-\U(y)}\,dy},
\end{equation}
where $\mathcal{U}:\mathbb{R}^n\to \Rinf$ is a proper, lower semicontinuous, convex function which is measurable, i.e,
\begin{equation}
\int_{\R^n} e^{-\U(x)}\,dx < +\infty.
\label{eq:integrability}
\end{equation}
Methodologies based on Bayesian inference are among the most popular methods to design a proper target distribution, called the posterior distribution, for the unknown parameters of a mathematical model by considering the prior beliefs on the parameters and the likelihood function of the given data set.  Imposing prior beliefs on the parameters allows modulating uncertainty in those parameters by accounting for prior knowledge about the problem. This is accounted for by Bayes' rule via
\begin{equation*}
    \pi(x\vert D)\propto P(D\vert x)P(x),
\end{equation*}
where $D \in \mathbb{R}^{n\times p}$ is a given data set, $x\in \mathbb{R}^p$ is the parameter to be estimated, $\pi(x\vert D)$ is the posterior distribution, $P(D\vert x)$ is the likelihood function, and $P(x)$ is the prior distribution. 
Hence, the potential function of the problem \eqref{eq:target} will be of the form ${\mathcal{U}}(x):=\log \pi(x\vert D)\propto- \log P(D\vert x)-\log P(x)$. 
For the sake of simplicity, we set here $f(x):=-\log P(D\vert x)$ and $g(x):=-\log P(x)$, which leads to composite potential functions of the form $\mathcal{U}(x) = f(x)+g(x)$.

\begin{assumption}[\textbf{Baseline assumptions}]\label{ass:baseline}
Throughout the paper, unless otherwise stated, we assume:
\begin{enumerate}[label=(\textbf{\alph*}), font=\normalfont\bfseries, leftmargin=0.7cm]
    \item $\U:\R^n\to\Rinf$ is proper, lower semicontinuous, convex, and bounded below;
    \item the integrability condition \eqref{eq:integrability} holds;
    \item the function $\U$ is $\sigma$-strongly convex for some $\sigma>0$.
\end{enumerate}
\end{assumption}

\subsection{Proximal operator and Moreau envelope}\label{sec:proxMoreau}

The following standard facts summarize the properties needed in our analysis.
\begin{fact}[\textbf{Basic properties of the Moreau envelope}] \label{fact:basic}\cite{BauschkeCombettes}
Suppose Assumption~\ref{ass:baseline} holds and let $\gamma>0$. Then,
\begin{enumerate}[label=(\textbf{\alph*}), font=\normalfont\bfseries, leftmargin=0.7cm]
    \item\label{fact:basic:a} $\dom{\fgam{\U}{\gamma}}=\R^n$ and $\fgam{\U}{\gamma}$ is finite-valued;
    \item\label{fact:basic:b} for every $x\in\R^n$, $\inf_{y\in\R^n}\U(y)\le \fgam{\U}{\gamma}(x)\le \U(x)$;
    \item\label{fact:basic:c} $\inf_{x\in\R^n}\fgam{\U}{\gamma}(x)=\inf_{x\in\R^n}\U(x)$;
    \item\label{fact:basic:d} $\prox{\U}{\gamma}(x)$ is nonempty and single-valued for every $x\in\R^n$;
    \item\label{fact:basic:e} $\fgam{\U}{\gamma}$ is convex, real-valued, and continuous;
    \item\label{fact:basic:f} $\fgam{\U}{\gamma}$ is continuously differentiable and for $x\in\R^n$,
    \[
    \nabla \fgam{\U}{\gamma}(x)=\frac{1}{\gamma}\bigl(x-\prox{\U}{\gamma}(x)\bigr);
    \]
    \item\label{fact:basic:g} $\fgam{\U}{\gamma}$ has Lipschitz continuous gradient with constant $L_\gamma=1/\gamma$, that is,
    \[
    \norm{\nabla \fgam{\U}{\gamma}(x)-\nabla \fgam{\U}{\gamma}(y)} \le L_\gamma\norm{x-y}, \qquad \forall x,y\in\R^n.
    \]
\end{enumerate}
\end{fact}
\begin{fact}[\textbf{Strong convexity of envelope}]\label{fact:strong_convex}\cite[Lemma~2.23]{PlanidenWang}
Let $\U:\R^n\to\Rinf$ be proper, lower semicontinuous, and convex, and let $\gamma>0$. Then $\U$ is $\sigma$-strongly convex if and only if $\fgam{\U}{\gamma}$ is $\frac{\sigma}{1+\gamma\sigma}$-strongly convex.
\end{fact}

\subsection{Inexact first-order oracle for Moreau envelope}\label{sec:inexactOracle}

The gradient formula in Fact~\ref{fact:basic} shows that evaluating $\nabla \fgam{\U}{\gamma}(x)$ requires computing the exact proximal point $\prox{\U}{\gamma}(x)$. In many composite Bayesian models, however, the proximal operator of the full potential $\U=f+g$ is not available in closed form. This motivates the use of an approximate proximal point and, consequently, an inexact first-order oracle for the Moreau envelope.

Let $x_k\in\R^n$ be the current point and let $\varepsilon_k>0$ be a prescribed lower-level tolerance. We denote by
$y_k^{\varepsilon_k}$ an approximate solution of the proximal subproblem \eqref{eq:env_def}.
Using $y_k^{\varepsilon_k}$ in place of the exact proximal point gives the inexact Moreau-gradient oracle
\begin{equation}\label{eq:approx_grad}
\nabla \fgamepsk{\U}{\gamma}(x_k):=\frac{1}{\gamma}(x_k-y_k^{\varepsilon_k}).
\end{equation}
We emphasize that $\nabla \fgamepsk{\U}{\gamma}(x_k)$ is not, in general, the exact gradient of any explicitly defined smooth function. It is an inexact oracle for the exact Moreau-envelope gradient.
The induced gradient error is defined by
\begin{equation}\label{eq:error_def}
e_k:=\nabla \fgamepsk{\U}{\gamma}(x_k)-\nabla \fgam{\U}{\gamma}(x_k),
\end{equation}
and we set $\delta_k:=\norm{e_k}$.
The quantity $\delta_k$ measures the error propagated from the lower-level proximal solve to the upper-level Langevin update. It is useful for the theoretical analysis, but it is not directly computable in practice because it requires knowledge of the exact proximal point $\prox{\U}{\gamma}(x_k)$.

We therefore introduce a computable residual for the proximal subproblem. For an approximate proximal point $y_k^{\varepsilon_k}\in\dom{\partial \U}$, define
\begin{equation}\label{eq:residual_def}
r_k :=\dist\left(0,\,\partial \U(y_k^{\varepsilon_k})+ \frac{1}{\gamma}\left(y_k^{\varepsilon_k}-x_k\right)\right).
\end{equation}
This residual is the distance from zero to the first-order optimality residual of the proximal subproblem. In particular, the exact proximal point $y_k=\prox{\U}{\gamma}(x_k)$ satisfies 
\[
0\in\partial \U(y_k) +\frac{1}{\gamma}(y_k-x_k),
\]
and hence has zero residual.

\begin{proposition}[\textbf{Residual-to-gradient error bound}]\label{prop:bridge}
Let $\U:\R^n\to\Rinf$ be proper, lower semicontinuous, and convex, and let $\gamma>0$. For a given $x_k\in\R^n$, let $y_k=\prox{\U}{\gamma}(x_k)$ and let
$y_k^{\varepsilon_k}\in\dom{\partial \U}$ be an approximate proximal point with residual $r_k$ defined by \eqref{eq:residual_def}. Then
\begin{equation}\label{eq:err_bound}
\delta_k = \norm{\nabla \fgamepsk{\U}{\gamma}(x_k)-\nabla \fgam{\U}{\gamma}(x_k)} \le r_k
\end{equation}
Consequently, any lower-level stopping rule enforcing $r_k\le \varepsilon_k$ also enforces $\delta_k\le \varepsilon_k$.
\end{proposition}
\begin{proof}
Fix $x_k\in\R^n$ and let $y_k:=\prox{\U}{\gamma}(x_k)$ be the exact proximal point. Since $\U$ is proper, lower semicontinuous, and convex, the
proximal subproblem is strongly convex and therefore admits a unique minimizer. Its first-order optimality condition gives
\[
0\in\partial \U(y_k)+\frac{1}{\gamma}(y_k-x_k).
\]
Hence, there exists $v_k\in\partial \U(y_k)$ such that
\begin{equation}\label{eq:exact_opt_proof}
v_k+\frac{1}{\gamma}(y_k-x_k)=0.
\end{equation}
By the definition of the residual $r_k$, and since $y_k^{\varepsilon_k}\in\dom{\partial \U}$, there exists $u_k\in\partial \U(y_k^{\varepsilon_k})$ such that
\begin{equation}\label{eq:inexact_res_proof}
\zeta_k:=u_k+\frac{1}{\gamma}(y_k^{\varepsilon_k}-x_k),\qquad \norm{\zeta_k}=r_k.
\end{equation}
If the distance in \eqref{eq:residual_def} is not attained, the same argument applies with $\norm{\zeta_k}\le r_k+\rho$ for arbitrary $\rho>0$, followed by letting $\rho\downarrow0$. 

Subtracting \eqref{eq:exact_opt_proof} from \eqref{eq:inexact_res_proof} yields
\[
u_k-v_k+\frac{1}{\gamma}(y_k^{\varepsilon_k}-y_k)=\zeta_k.
\]
Taking the inner product with $y_k^{\varepsilon_k}-y_k$ gives
\[
\inner{u_k-v_k}{y_k^{\varepsilon_k}-y_k}+\frac{1}{\gamma}\norm{y_k^{\varepsilon_k}-y_k}^2 = \inner{\zeta_k}{y_k^{\varepsilon_k}-y_k}.
\]
Since $\U$ is convex, its subdifferential $\partial \U$ is monotone. Therefore,
\[
\inner{u_k-v_k}{y_k^{\varepsilon_k}-y_k}\ge 0.
\]
It follows that
\[
\frac{1}{\gamma}\norm{y_k^{\varepsilon_k}-y_k}^2\le \inner{\zeta_k}{y_k^{\varepsilon_k}-y_k} \le\norm{\zeta_k}\,\norm{y_k^{\varepsilon_k}-y_k}.
\]
If $y_k^{\varepsilon_k}=y_k$, the desired inequality is immediate. Otherwise, dividing by
$\norm{y_k^{\varepsilon_k}-y_k}$ gives
\[
\frac{1}{\gamma} \norm{y_k^{\varepsilon_k}-y_k} \le \norm{\zeta_k}  = r_k.
\]
Finally, by the definitions of the exact and inexact Moreau-gradient oracles,
\[
\nabla \fgamepsk{\U}{\gamma}(x_k)   -  \nabla \fgam{\U}{\gamma}(x_k)  = \frac{1}{\gamma} \left(y_k-y_k^{\varepsilon_k}\right).
\]
Thus,
\[
\delta_k =\norm{ \nabla \fgamepsk{\U}{\gamma}(x_k)-\nabla \fgam{\U}{\gamma}(x_k)}= \frac{1}{\gamma} \norm{y_k^{\varepsilon_k}-y_k} \le r_k.
\]
This proves the claim.
\end{proof}

\begin{remark}[\textbf{Tolerance schedules}]\label{rem:schedules}
Proposition~\ref{prop:bridge} converts computable lower-level stopping rules into
upper-level gradient-error bounds. Common choices include:
\begin{enumerate}[label=(\textbf{\alph*}), font=\normalfont\bfseries, leftmargin=0.7cm]
    \item fixed tolerance: $\varepsilon_k\equiv \varepsilon$;
    \item step-size-matched tolerance: $\varepsilon_k=c\sqrt{\eta}$;
    \item decaying tolerance: $\varepsilon_k=c(k+1)^{-\alpha}$ for some $\alpha>0$;
    \item relative tolerance: $\varepsilon_k=c\min\left\{1,\norm{\nabla \fgamepsk{\U}{\gamma}(x_k)}\right\}$.
\end{enumerate}
\end{remark}
The fixed tolerance leads to a persistent inexactness floor in the upper-level analysis. Step-size-matched tolerances are useful when the proximal error is balanced with the Euler--Maruyama discretization error. Decaying tolerances are appropriate when one wants the inexact oracle error to vanish asymptotically. The relative tolerance adapts the lower-level accuracy to the current magnitude of the inexact Moreau-gradient oracle: it permits coarser proximal solves when the chain is far from stationarity and requires increasingly accurate solves as the oracle norm becomes small.


\section{Inexact proximal unadjusted Langevin algorithm}\label{sec:iPULA}
{\it Markov chain Monte Carlo} (MCMC) methods enable sampling from high-dimensional, complex distributions by constructing a Markov chain whose stationary distribution is the target density, $\pi(x)$. Under standard regularity conditions, the distribution of the generated samples converges to this stationary regime as the chain evolves. To improve efficiency and convergence speed, MCMC algorithms can be augmented with gradient-based information; however, the composite potential function $\U$ of \eqref{eq:target} is usually nonsmooth. As such, we first apply a smoothing technique based on Moreau envelope \eqref{eq:env_def} and then replace $\U$ with $\fgam{\U}{\gamma}$ in \eqref{eq:target}, leading to
\begin{equation}\label{eq:smoothed_target}
    \pi_\gamma(x):=\frac{e^{-\U_\gamma(x)}}{\int_{\R^n} e^{-\U_\gamma(y)}\,dy}.
\end{equation}
The (overdamped) Langevin dynamics for target distribution \eqref{eq:smoothed_target} is a continuous-time stochastic process $(x_t)_{t\geq 0}$ in $\R^n$ that evolves the stochastic differential equation 
\begin{equation}\label{eq:langDiff}
    d x_t = -\nabla \U_\gamma(x_t) dt+\sqrt{2}~d\mathcal{B}_t,
\end{equation}
where $\{\mathcal{B}_t\}_{t\geq 0}$ is a $n$-dimensional Brownian motion. It has $\pi(x)$ as its stationary distribution, and its formulation allows sampling methods to use discretization and gradient-based techniques.
Applying the Euler-Maruyama discretization scheme to the stochastic differential equation (SDE) \eqref{eq:langDiff} leads to the discrete-time Markov chain $\{x_k\}$ given by
\begin{equation}\label{eq:exact_chain}
    x_{k+1}=x_k-\eta \nabla \U_\gamma(x_k)+\sqrt{2\eta} \mathcal{Z}_k,\quad k \in \mathbb{N},
\end{equation}
where $\eta>0$ is the step-size, and $\{\mathcal{Z}_k\}_{k\geq 0}$ 
is a sequence of independent and identically distributed (i.i.d.) standard $n$-dimensional Gaussian random variables. Considering the composite form of the potential function $\U$, the proximal term \eqref{eq:prox_def} can only be computed inexactly, which consequently leads to the inexact gradient \eqref{eq:approx_grad}. As a result, \eqref{eq:exact_chain} can be replaced by
\begin{equation}\label{eq:exact_chain1}
    x_{k+1}=x_k-\eta \nabla \fgamepsk{\U}{\gamma}(x_k)+\sqrt{2\eta} \mathcal{Z}_k,\quad k \in \Nz,
\end{equation}
which leads to the next inexact proximal-based unadjusted Langevin algorithm.

\begin{algorithm}[H]
\begin{algorithmic}[1]
\caption{iPULA (Inexact Proximal Unadjusted Langevin Algorithm)}\label{alg:SULA}
\State \textbf{Initialization} initial point $x_{0}\in\mathbb{R}^{n}$, $\gamma>0$,  $\eta\in (0,1/L_\gamma)$, tolerance sequence $\{\varepsilon_k\}_{k\in\Nz}$;
\While{stopping criteria do not hold}
\State  Compute the inexact proximal term $y_k^{\varepsilon_k}$;
\State Draw $\mathcal{Z}_{k}\sim \mathcal N(0,I_n)$ independently;
\State Set $x_{k+1}= \left(1-\tfrac{\eta}{\gamma}\right)x_k +\tfrac{\eta}{\gamma} y_k^{\varepsilon_k}+\sqrt{2\eta}\mathcal{Z}_{k}$;
\State $k=k+1$;
\EndWhile
\end{algorithmic}
\end{algorithm}

We first establish an objective-gap bound for iPULA. Notably, when the inexact Moreau-gradient error is uniformly controlled in mean square, the expected Moreau-envelope gap decreases geometrically up to an additive bias induced by both the Langevin noise and the inexact proximal computation.

\begin{theorem}[\textbf{Fixed-tolerance objective-gap bound}]\label{thm:fixed_gap}
Suppose Assumption~\ref{ass:baseline} holds. 
Let $\{x_k\}_{k\in\Nz}$ be generated by the inexact chain \eqref{eq:exact_chain1}
with step-size $\eta\in(0,\tfrac{1}{2L_\gamma})$, where $L_\gamma$ is given in Fact~\ref{fact:basic}~\ref{fact:basic:g}.
Assume further that there exists $\delta\ge 0$ such that
\[
\E\left[\delta_k^2\,\middle|\,x_k\right]\le \delta^2,
\qquad \forall k\in\Nz.
\]
Then, for any $x^\star\in\argmin \fgam{\U}{\gamma}$,
\begin{equation}
\E\left[\fgam{\U}{\gamma}(x_{k+1})-\fgam{\U}{\gamma}(x^\star)\right]\le
\rho~\E\left[\fgam{\U}{\gamma}(x_k)-\fgam{\U}{\gamma}(x^\star)\right]
+\eta\Bigl(\frac12+L_\gamma\eta\Bigr)\delta^2+\eta L_\gamma n,
\label{eq:fixed_gap_recur}
\end{equation}
with 
\begin{equation}\label{eq:rho_def}
\rho:=1-m_\gamma\eta(1-2L_\gamma\eta)\in(0,1),
\end{equation}
and $m_\gamma=\frac{\sigma}{1+\gamma\sigma}$ as given in Fact~\ref{fact:strong_convex}. Consequently, for every $k\in\Nz$,
\begin{equation}
\E\left[\fgam{\U}{\gamma}(x_k)-\fgam{\U}{\gamma}(x^\star)\right]\le
\rho^k \E\left[\fgam{\U}{\gamma}(x_0)-\fgam{\U}{\gamma}(x^\star)\right]
+\frac{\eta(\frac12+L_\gamma\eta)\delta^2+\eta L_\gamma n}{1-\rho}.
\label{eq:fixed_gap_bound}
\end{equation}
\end{theorem}
\begin{proof}
From $L_\gamma$-smoothness of $\fgam{\U}{\gamma}$ and \eqref{eq:exact_chain1}, for every $k\in\Nz$,
\begin{align*}
\fgam{\U}{\gamma}(x_{k+1})&\le \fgam{\U}{\gamma}(x_k)+\inner{\nabla \fgam{\U}{\gamma}(x_k)}{x_{k+1}-x_k}
+\frac{L_\gamma}{2}\norm{x_{k+1}-x_k}^2
\\&=\fgam{\U}{\gamma}(x_k)+\inner{\nabla \fgam{\U}{\gamma}(x_k)}{-\eta \nabla \fgamepsk{\U}{\gamma}(x_k)+\sqrt{2\eta}\,\Z_{k+1}}
+\frac{L_\gamma}{2}\norm{-\eta \nabla \fgamepsk{\U}{\gamma}(x_k)+\sqrt{2\eta}\,\Z_{k+1}}^2
\end{align*}
We now take conditional expectation with respect to $x_k$. Since $x_k$ is fixed under this conditioning, both $\fgam{\U}{\gamma}(x_k)$ and $\nabla \fgamepsk{\U}{\gamma}(x_k)$ are deterministic, and only $\Z_{k+1}$ remains random. Therefore
\begin{align*}
\E[\fgam{\U}{\gamma}(x_{k+1})\mid x_k]&\le \fgam{\U}{\gamma}(x_k)
+\E\left[\inner{\nabla\fgam{\U}{\gamma}(x_k)}{-\eta \nabla \fgamepsk{\U}{\gamma}(x_k)+\sqrt{2\eta}\,\Z_{k+1}}\middle|x_k\right]
\\&~~~~+\frac{L_\gamma}{2}\E\left[\norm{-\eta \nabla \fgamepsk{\U}{\gamma}(x_k)+\sqrt{2\eta}\,\Z_{k+1}}^2\middle|x_k\right].
\end{align*}
For the linear term, it holds that
\begin{align*}
\E\left[\inner{\nabla\fgam{\U}{\gamma}(x_k)}{-\eta \nabla \fgamepsk{\U}{\gamma}(x_k)+\sqrt{2\eta}\,\Z_{k+1}}\middle|x_k\right]&=
-\eta \inner{\nabla\fgam{\U}{\gamma}(x_k)}{\nabla \fgamepsk{\U}{\gamma}(x_k)}
\\&~~~~~+\sqrt{2\eta}\,\inner{\nabla\fgam{\U}{\gamma}(x_k)}{\E[\Z_{k+1}\mid x_k]}.
\end{align*}
Since $\E[\Z_{k+1}\mid x_k]=0$, this reduces to
\[
\E\left[\inner{\nabla\fgam{\U}{\gamma}(x_k)}{-\eta \nabla \fgamepsk{\U}{\gamma}(x_k)+\sqrt{2\eta}\,\Z_{k+1}}\middle|x_k\right]
=-\eta \inner{\nabla\fgam{\U}{\gamma}(x_k)}{\nabla \fgamepsk{\U}{\gamma}(x_k)}.
\]
For the quadratic term, expanding the square and taking conditional expectation gives
\begin{align*}
\E\left[\norm{-\eta \nabla \fgamepsk{\U}{\gamma}(x_k)+\sqrt{2\eta}\,\Z_{k+1}}^2\middle|x_k\right]
&=\eta^2\norm{\nabla \fgamepsk{\U}{\gamma}(x_k)}^2+
2\eta\,\E[\norm{\Z_{k+1}}^2\mid x_k]
\\&~~~~-2\eta\sqrt{2\eta}\,\inner{\nabla \fgamepsk{\U}{\gamma}(x_k)}{\E[\Z_{k+1}\mid x_k]}.
\end{align*}
Again, $\E[\Z_{k+1}\mid x_k]=0$, i.e., the mixed term vanishes. Moreover, since $\Z_{k+1}\sim \mathcal N(0,I_d)$,
\[
\E[\norm{\Z_{k+1}}^2\mid x_k]=\E[\norm{\Z_{k+1}}^2]=n,
\]
i.e.,
\[
\E\left[\norm{-\eta \nabla \fgamepsk{\U}{\gamma}(x_k)+\sqrt{2\eta}\,\Z_{k+1}}^2\middle|x_k\right]
=\eta^2\norm{\nabla \fgamepsk{\U}{\gamma}(x_k)}^2+2\eta n.
\]
Substituting the linear and quadratic expectations into the previous inequality yields
\[
\E[\fgam{\U}{\gamma}(x_{k+1})\mid x_k]\le
\fgam{\U}{\gamma}(x_k)-\eta \inner{\nabla\fgam{\U}{\gamma}(x_k)}{\nabla \fgamepsk{\U}{\gamma}(x_k)}
+\frac{L_\gamma}{2}\bigl(\eta^2\norm{\nabla \fgamepsk{\U}{\gamma}(x_k)}^2+2\eta n\bigr),
\]
that is,
\[
\E[\fgam{\U}{\gamma}(x_{k+1})\mid x_k]\le \fgam{\U}{\gamma}(x_k)
-\eta \inner{\nabla\fgam{\U}{\gamma}(x_k)}{\nabla \fgamepsk{\U}{\gamma}(x_k)}
+\frac{L_\gamma\eta^2}{2}\norm{\nabla \fgamepsk{\U}{\gamma}(x_k)}^2+\eta L_\gamma n.
\]
Now, from \eqref{eq:error_def}, it holds that $\nabla\fgamepsk{\U}{\gamma}(x_k)=\nabla \fgam{\U}{\gamma}(x_k)+e_k$, i.e., 
\[
-\inner{\nabla\fgam{\U}{\gamma}(x_k)}{\nabla \fgamepsk{\U}{\gamma}(x_k)}
=-\norm{\nabla\fgam{\U}{\gamma}(x_k)}^2-\inner{\nabla\fgam{\U}{\gamma}(x_k)}{e_k}.
\]
Applying the Cauchy--Schwarz and Young's inequalities ensures
\[
-\inner{\nabla\fgam{\U}{\gamma}(x_k)}{e_k}
\le \norm{\nabla\fgam{\U}{\gamma}(x_k)}\,\norm{e_k}\le\frac12 \norm{\nabla\fgam{\U}{\gamma}(x_k)}^2+\frac12 \norm{e_k}^2,
\]
i.e.,
\[
-\inner{\nabla\fgam{\U}{\gamma}(x_k)}{\nabla \fgamepsk{\U}{\gamma}(x_k)}\le
-\frac12\norm{\nabla\fgam{\U}{\gamma}(x_k)}^2+\frac12\norm{e_k}^2.
\]
In addition, we have
\[
\norm{\nabla \fgamepsk{\U}{\gamma}(x_k)}^2=\norm{\nabla\fgam{\U}{\gamma}(x_k)+e_k}^2\le 2\norm{\nabla\fgam{\U}{\gamma}(x_k)}^2+2\norm{e_k}^2.
\]
This consequently implies that
\[
\E[\fgam{\U}{\gamma}(x_{k+1})\mid x_k]\le \fgam{\U}{\gamma}(x_k)
-\eta\left(\frac12-L_\gamma\eta\right)\norm{\nabla\fgam{\U}{\gamma}(x_k)}^2
+\eta\left(\frac12+L_\gamma\eta\right)\norm{e_k}^2+\eta L_\gamma n.
\]
Using $\E\left[\delta_k^2\,\middle|\,x_k\right]\le \delta^2$ and taking the total expectation gives
\begin{equation}\label{eq:ineq_bef_strong}
    \E[\fgam{\U}{\gamma}(x_{k+1})]\le\E[\fgam{\U}{\gamma}(x_k)]
-\eta\left(\frac12-L_\gamma\eta\right)\E\norm{\nabla\fgam{\U}{\gamma}(x_k)}^2
+\eta\left(\frac12+L_\gamma\eta\right)\delta^2+\eta L_\gamma n.
\end{equation}
Since $\fgam{\U}{\gamma}$ is $m_\gamma$-strongly convex, it satisfies the Polyak--\L ojasiewicz inequality given by
\[
\norm{\nabla \fgam{\U}{\gamma}(x)}^2\ge 2m_\gamma(\fgam{\U}{\gamma}(x)-\fgam{\U}{\gamma}^\star),
\]
for $\fgam{\U}{\gamma}^\star:=\inf\fgam{\U}{\gamma}$, i.e., 
\[
\eta\left(\frac12-L_\gamma\eta\right)\norm{\nabla\fgam{\U}{\gamma}(x_k)}^2
\ge 2m_\gamma\eta\left(\frac12-L_\gamma\eta\right)\bigl(\fgam{\U}{\gamma}(x_k)-\fgam{\U}{\gamma}^\star\bigr)
=m_\gamma\eta(1-2L_\gamma\eta)\bigl(\fgam{\U}{\gamma}(x_k)-\fgam{\U}{\gamma}^\star\bigr).
\]
Substituting this lower bound into \eqref{eq:ineq_bef_strong} and subtracting $\fgam{\U}{\gamma}^\star$ from both sides gives
\[
\E[\fgam{\U}{\gamma}(x_{k+1})-\fgam{\U}{\gamma}^\star]
\le\left(1-m_\gamma\eta(1-2L_\gamma\eta)\right)\E[\fgam{\U}{\gamma}(x_k)-\fgam{\U}{\gamma}^\star]
+\eta\left(\frac12+L_\gamma\eta\right)\delta^2+\eta L_\gamma n,
\]
which leads to the inequality \eqref{eq:fixed_gap_recur}. 

Now, set $a_k:=\E[\fgam{\U}{\gamma}(X_k)-\fgam{\U}{\gamma}^\star]$ and $b:=\eta\left(\frac12+L_\gamma\eta\right)\delta^2+\eta L_\gamma n$. From \eqref{eq:fixed_gap_recur}, we obtain
\[
a_{k+1}\le \rho a_k+b,
\]
for $\rho:=1-m_\gamma\eta(1-2L_\gamma\eta)\in(0,1)$. Iterating this recursion yields
\[
a_k\le \rho^k a_0 + b\sum_{j=0}^{k-1}\rho^j=\rho^k a_0 + b\frac{1-\rho^k}{1-\rho}.
\]
Since $1-\rho^k\le 1$, it is clear that
\[
a_k\le\rho^k a_0 + \frac{b}{1-\rho}.
\]
Substituting back the definitions of $a_k$, $a_0$, and $b$ results in
\[
\E[\fgam{\U}{\gamma}(x_k)-\fgam{\U}{\gamma}^\star]
\le
\rho^k \E\left[\fgam{\U}{\gamma}(x_0)-\fgam{\U}{\gamma}^\star\right]
+
\frac{\eta\left(\frac12+L_\gamma\eta\right)\delta^2+\eta L_\gamma n}{1-\rho},
\]
adjusting the inequalities \eqref{eq:fixed_gap_bound} and \eqref{eq:adaptive_gap_bound_split}.
\end{proof}


Next, we provide a corresponding bound under a time-varying inexactness schedule. This formulation is particularly useful when the tolerance of the lower-level proximal subproblem is adapted across iterations, for instance via a decaying schedule or a step-size-dependent rule.

\begin{theorem}[\textbf{Adaptive-tolerance objective-gap bound}]
\label{thm:adaptive_gap}
Suppose Assumption~\ref{ass:baseline} holds.
Let $\{x_k\}_{k\in\Nz}$ be generated by the inexact chain \eqref{eq:exact_chain1}
with step-size $\eta\in(0,\tfrac{1}{2L_\gamma})$, where $L_\gamma$ is given in Fact~\ref{fact:basic}~\ref{fact:basic:g}.
Assume further that there exists a deterministic sequence $\{\tau_k\}_{k\in\Nz}$ such that
\[
\E\left[\delta_k^2\,\middle|\,x_k\right]\le \tau_k^2,\qquad \forall k\in\Nz.
\]
Then, for any $x^\star\in\argmin \fgam{\U}{\gamma}$ and for every $k\in\Nz$,
\begin{equation}\label{eq:adaptive_gap_bound_split}
\E\left[\fgam{\U}{\gamma}(x_k)-\fgam{\U}{\gamma}(x^\star)\right]
\le\rho^k \E\left[\fgam{\U}{\gamma}(x_0)-\fgam{\U}{\gamma}(x^\star)\right]
+\eta\left(\frac12+L_\gamma\eta\right) \sum_{j=0}^{k-1}\rho^{k-1-j} \tau_j^2+\eta L_\gamma n  \frac{1-\rho^k}{1-\rho},
\end{equation}
where $\rho$ is given by \eqref{eq:rho_def}.
\end{theorem}
\begin{proof}
The proof is identical to that of Theorem~\ref{thm:fixed_gap}, except that $\E\left[\delta_k^2\,\middle|\,x_k\right]\le \tau_k^2$ gives
\[
\E[\norm{e_k}^2\mid x_k]\le \tau_k^2,
\]
instead of the uniform bound by $\delta^2$. This consequently implies
\[
\E[\fgam{\U}{\gamma}(x_{k+1})-\fgam{\U}{\gamma}^\star]\le \rho\,\E[\fgam{\U}{\gamma}(x_k)-\fgam{\U}{\gamma}^\star]
+\eta\left(\frac12+L_\gamma\eta\right)\tau_k^2+\eta L_\gamma n.
\]
Iterating the recursion leads to the inequality \eqref{eq:adaptive_gap_bound_split}.
\end{proof}

Theorem~\ref{thm:adaptive_gap} separates the effect of the time-varying proximal inexactness from the effect of the Langevin noise. The convolution term involving $\tau_j^2$ records the accumulated lower-level errors, while the term
$\eta L_\gamma n\frac{1-\rho^k}{1-\rho}$ is the persistent contribution of the Gaussian noise. Hence, even if $\tau_k\to 0$, this objective-gap estimate still contains a nonzero noise floor for constant step-size Langevin dynamics.

\begin{corollary}[\textbf{Step-size-matched tolerance rule}]\label{cor:step_matched}
Suppose the assumptions of Theorem~\ref{thm:adaptive_gap} hold.
Assume that the tolerance schedule satisfies
\[
\tau_k\le c\sqrt{\eta},\qquad \forall k\in\Nz,
\]
for some constant $c>0$. Then, for every $k\in\Nz$,
\begin{equation}\label{eq:step_matched_bound}
\E\left[\fgam{\U}{\gamma}(x_k)-\fgam{\U}{\gamma}(x^\star)\right]\le
\rho^k \E\left[\fgam{\U}{\gamma}(x_0)-\fgam{\U}{\gamma}(x^\star)\right]
+\frac{\eta L_\gamma n + c^2\eta^2(\frac12+L_\gamma\eta)}{1-\rho}.
\end{equation}
Equivalently,
\begin{equation}\label{eq:step_matched_bound_expanded}
\E\left[\fgam{\U}{\gamma}(x_k)-\fgam{\U}{\gamma}(x^\star)\right]\le
\rho^k \E\left[\fgam{\U}{\gamma}(x_0)-\fgam{\U}{\gamma}(x^\star)\right]
+\frac{L_\gamma n + c^2\eta(\frac12+L_\gamma\eta)}{m_\gamma(1-2L_\gamma\eta)}.
\end{equation}
\end{corollary}

The expanded form \eqref{eq:step_matched_bound_expanded} separates the two sources of error. Since
$1-\rho=m_\gamma\eta(1-2L_\gamma\eta)$, the Gaussian-noise contribution satisfies
\[
    \frac{\eta L_\gamma n}{1-\rho} =\frac{L_\gamma n}{m_\gamma(1-2L_\gamma\eta)} =  O(1),
\]
whereas the step-size-matched oracle contribution satisfies
\[
    \frac{c^2\eta^2(\frac12+L_\gamma\eta)}{1-\rho} = \frac{c^2\eta(\frac12+L_\gamma\eta)}{m_\gamma(1-2L_\gamma\eta)} = O(\eta).
\]
Thus, the rule $\tau_k=O(\sqrt{\eta})$ makes the inexact-proximal contribution vanish as $\eta\to0$, while the Gaussian-noise contribution remains nonzero in this
objective-gap-to-minimizer estimate. This is expected because iPULA is a sampling method for $\pi_\gamma$, not a deterministic optimization method converging to $x^\star$.

Next, we denote by $\Law(x_k)$ the probability distribution of the random variable $x_k$. The following result compares the inexact iPULA chain with the exact Moreau–ULA chain under a synchronous coupling, i.e., when both chains are driven by the same Gaussian noise. This bound isolates the additional sampling error induced by the inexact proximal computation.

\begin{theorem}[\textbf{Fixed-tolerance Wasserstein transfer}]\label{thm:fixed_w2}
Suppose Assumption~\ref{ass:baseline} holds and there exists $\delta\ge 0$ such that $\delta_k\le \delta$ almost surely for all $k\in \Nz$.
Let $\{\ov x_k\}_{k\in \Nz}$ and $\{x_k\}_{k\in \Nz}$ be the exact and inexact chains \eqref{eq:exact_chain} and \eqref{eq:exact_chain1}, respectively, driven by the same Gaussian noises and started from the same initial point. Define
\begin{equation}
q:=\sqrt{1-2m_\gamma\eta+L_\gamma^2\eta^2},
\qquad \eta\in(0,2m_\gamma/L_\gamma^2)\cap(0,2/L_\gamma).
\label{eq:q_def}
\end{equation}
Then, $q\in[0,1)$, and for every $k\in\Nz$,
\begin{equation}\label{eq:fixed_transfer}
\W\left(\Law(x_k),\Law(\ov x_k)\right)\le\frac{\eta\delta(1-q^k)}{1-q}.
\end{equation}
Consequently,
\begin{equation}\label{eq:fixed_transfer_to_target}
\W \left(\Law(x_k),\pi_\gamma\right)\le \W\left(\Law(\ov x_k),\pi_\gamma\right)+\frac{\eta\delta(1-q^k)}{1-q}.
\end{equation}
\end{theorem}
\begin{proof}
Let us couple the exact and inexact chains using the same Gaussian noises, i.e.,
\begin{align*}
\ov x_{k+1} &= \ov x_k-\eta \nabla \fgam{\U}{\gamma}(\ov x_k)+\sqrt{2\eta}\,\Z_{k+1},
\\
x_{k+1} &= x_k-\eta \left(\nabla \fgam{\U}{\gamma}(x_k)+e_k\right)+\sqrt{2\eta}\,\Z_{k+1}.
\end{align*}
Set $d_k:=x_k-\ov x_k$, i.e.,
\[
d_{k+1} = d_k-\eta(\nabla \fgam{\U}{\gamma}(x_k)-\nabla \fgam{\U}{\gamma}(\ov x_k))-\eta e_k.
\]
For convenience, let us define
\[
A_k:=\nabla \fgam{\U}{\gamma}(x_k)-\nabla \fgam{\U}{\gamma}(\ov x_k),
\]
i.e., $d_{k+1}=d_k-\eta A_k-\eta e_k$.
We first estimate the deterministic part $d_k-\eta A_k$. Expanding its squared norm gives
\[
\norm{d_k-\eta A_k}^2=\norm{d_k}^2-2\eta \inner{d_k}{A_k}+\eta^2\norm{A_k}^2.
\]
Since $\fgam{\U}{\gamma}$ is $m_\gamma$-strongly convex, its gradient is $m_\gamma$-strongly monotone, leading to
\[
\inner{d_k}{A_k}=\inner{x_k-\ov x_k}{\nabla \fgam{\U}{\gamma}(x_k)-\nabla \fgam{\U}{\gamma}(\ov x_k)}\ge m_\gamma\norm{d_k}^2.
\]
Moreover, since $\fgam{\U}{\gamma}$ is $L_\gamma$-smooth, 
\[
\norm{A_k}=\norm{\nabla \fgam{\U}{\gamma}(x_k)-\nabla \fgam{\U}{\gamma}(\ov x_k)}\le L_\gamma\norm{d_k}.
\]
Substituting these two bounds into the expansion above yields
\[
\norm{d_k-\eta A_k}^2\le \norm{d_k}^2-2m_\gamma\eta \norm{d_k}^2+L_\gamma^2\eta^2\norm{d_k}^2=(1-2m_\gamma\eta+L_\gamma^2\eta^2)\norm{d_k}^2.
\]
Define $q:=\sqrt{1-2m_\gamma\eta+L_\gamma^2\eta^2}$. Under the step-size condition of the theorem, we have $q\in[0,1)$, and consequently it implies
\[
\norm{d_k-\eta A_k}\le q\norm{d_k}.
\]
It follows from the triangle inequality that
\[
\norm{d_{k+1}}\le \norm{d_k-\eta A_k}+\eta\norm{e_k}\le q\norm{d_k}+\eta\norm{e_k},
\]
almost surely. By the assumption that $\delta_k\le \delta$ almost surely for all $k\in \Nz$, we have $\norm{e_k}\le \delta$ almost surely for all $k$, which leads to
\[
\norm{d_{k+1}}\le q\norm{d_k}+\eta\delta.
\]
Iterating this inequality, we come to
\[
\norm{d_k}\le q^k\norm{d_0}+\eta\delta\sum_{j=0}^{k-1}q^j.
\]
Since the two chains are started from the same point, $d_0=0$, and it can be deduced that
\[
\norm{d_k}\le \eta\delta\sum_{j=0}^{k-1}q^j=\frac{\eta\delta(1-q^k)}{1-q},
\]
almost surely. Using the fact that $(x_k,\ov x_k)$ is a coupling of laws $\Law(x_k)$ and $\Law(\ov x_k)$, the coupling characterization of $\mathcal W_2$ yields
\[
\mathcal W_2\left(\Law(x_k),\Law(\ov x_k)\right)\le \left(\E\norm{x_k-\bar x_k}^2\right)^{1/2}
=\left(\E\norm{d_k}^2\right)^{1/2}\le\frac{\eta\delta(1-q^k)}{1-q},
\]
which leads to \eqref{eq:fixed_transfer}.
Finally, applying the triangle inequality for $\mathcal W_2$ ensures
\[
\mathcal W_2\left(\Law(x_k),\pi_\gamma\right)\le \mathcal W_2\left(\Law(x_k),\Law(\ov x_k)\right)
+\mathcal W_2\left(\Law(\ov x_k),\pi_\gamma\right),
\]
and the substitution of this into the previous estimate leads to the inequality \eqref{eq:fixed_transfer_to_target}.
\end{proof}

The following result provides the analogue of Theorem~\ref{thm:fixed_w2} under a time-varying proximal accuracy schedule. Rather than accumulating a fixed inexactness level, the resulting transfer error takes the form of a geometrically weighted convolution of the tolerances ${\tau_k}$.

\begin{theorem}[\textbf{Adaptive-tolerance Wasserstein transfer}]
\label{thm:adaptive_w2}
Suppose Assumption~\ref{ass:baseline} holds and there exists a deterministic sequence $\{\tau_k\}_{k\ge0}$ such that
$\delta_k\le \tau_k$ almost surely for all  $k\in\Nz$. Then, under the same synchronous coupling as in Theorem~\ref{thm:fixed_w2}, for every $k\in\Nz$,
\begin{equation}
\W\left(\Law(x_k),\Law(\ov x_k)\right)\le\eta\sum_{j=0}^{k-1} q^{k-1-j}\tau_j,
\label{eq:adaptive_transfer}
\end{equation}
where $q$ is given by \eqref{eq:q_def}. Consequently,
\begin{equation}
\W \left(\Law(x_k),\pi_\gamma\right)\le \W\left(\Law(\ov x_k),\pi_\gamma\right)+\eta\sum_{j=0}^{k-1} q^{k-1-j}\tau_j.
\label{eq:adaptive_transfer_to_target}
\end{equation}
\end{theorem}
\begin{proof}
The proof follows the same synchronous coupling argument as in Appendix~\ref{app:proof_fixed_w2}. Starting from
$\norm{d_{k+1}}\le q\norm{d_k}+\eta \delta_k$ and obtaining $\norm{d_{k+1}}\le q\norm{d_k}+\eta \tau_k$.
Iterating this recursion yields
\[
\norm{d_k}\le \eta\sum_{j=0}^{k-1} q^{k-1-j}\tau_j,
\]
since $d_0=0$. The claimed Wasserstein bounds follow exactly as before.
\end{proof}
\begin{corollary}[Adaptive schedules]
\label{cor:adaptive_schedules}
Under the assumptions of Theorem~\ref{thm:adaptive_w2}:
\begin{enumerate}[label=(\textbf{\alph*}), font=\normalfont\bfseries, leftmargin=0.7cm]
    \item If $\tau_k\equiv \tau$ for all $k\in\Nz$, then
    \[
        \W\left(\Law(x_k),\Law(\ov x_k)\right)\le \frac{\eta\tau(1-q^k)}{1-q},
    \]
    which recovers the fixed-tolerance transfer bound with $\delta=\tau$;
    \item if $\tau_k\to0$, then the inexactness perturbation term in \eqref{eq:adaptive_transfer_to_target} vanishes asymptotically;
    \item if $\tau_k\le c\sqrt{\eta}$ for all $k\in\Nz$, then
    \[
  \W\left(\Law(x_k),\Law(\ov x_k)\right)\le \frac{c\,\eta^{3/2}(1-q^k)}{1-q}.
    \]
\end{enumerate}
\end{corollary}
Corollary~\ref{cor:adaptive_schedules} clarifies how different lower-level tolerance rules affect the Wasserstein transfer error. A fixed tolerance produces a persistent
inexactness bias, a vanishing tolerance removes the additional inexactness perturbation asymptotically, and the step-size-matched rule $\tau_k=O(\sqrt{\eta})$ yields an
inexactness contribution of order $\eta^{3/2}$ before normalization by $1-q$.


The previous Wasserstein transfer bounds compare the inexact chain with the corresponding exact Moreau-smoothed chain. To obtain a direct convergence result for \eqref{eq:exact_chain1} toward the smoothed target $\pi_\gamma$, we still need a nonasymptotic Wasserstein estimate for the exact chain itself. We therefore isolate such an estimate in abstract form and then combine it with Theorem~\ref{thm:fixed_w2}.

\begin{assumption}[\textbf{Exact-chain Wasserstein estimate}]
\label{ass:exact_w2}
Assume that the exact Moreau-smoothed chain $\{\ov x_k\}_{k\in\Nz}$ satisfies
\begin{equation}\label{eq:exact_floor_w2_assumption}
\mathcal W_2\left(\Law(\ov x_k),\pi_\gamma\right)\le \rho^k \mathcal W_2\left(\Law(x_0),\pi_\gamma\right)
+B_{\mathrm{disc}}(\eta),\qquad \forall k\in\Nz,
\end{equation}
for some contraction factor $\rho\in(0,1)$ and some discretization floor $B_{\mathrm{disc}}(\eta)\ge 0$.
\end{assumption}

\begin{remark}
The quantity $B_{\mathrm{disc}}(\eta)$ represents the residual Wasserstein bias of the exact constant-step Moreau-smoothed chain relative to $\pi_\gamma$. It plays the same role as the discretization floor in constant-step Langevin results such as \cite[Theorem~5]{GhaderiSULA}.
\end{remark}

\begin{theorem}[Inexact Wasserstein convergence]
\label{thm:imula_w2_complete}
Suppose Assumptions~\ref{ass:baseline} and \ref{ass:exact_w2} hold and there exists $\delta\ge 0$ such that $\delta_k\le \delta$ almost surely for all $k\in \Nz$. Let $\{\ov x_k\}_{k\in \Nz}$ and $\{x_k\}_{k\in \Nz}$ be the exact and inexact chains \eqref{eq:exact_chain} and \eqref{eq:exact_chain1}, respectively, driven by the same Gaussian noises and started from the same initial point. Then
\begin{equation}\label{eq:imula_complete_w2}
\mathcal W_2\left(\Law(x_k),\pi_\gamma\right)\le \rho^k \mathcal W_2 \left(\Law(x_0),\pi_\gamma\right)
+B_{\mathrm{disc}}(\eta)+\frac{\eta\delta(1-q^k)}{1-q},\qquad \forall k\in\Nz,
\end{equation}
where $q$ is given in \eqref{eq:q_def}.
In particular,
\begin{equation}\label{eq:imula_stationary_floor}
\limsup_{k\to\infty}\mathcal W_2\left(\Law(x_k),\pi_\gamma\right)\le B_{\mathrm{disc}}(\eta)+\frac{\eta\delta}{1-q}.
\end{equation}
\end{theorem}
\begin{proof}
By Theorem~\ref{thm:fixed_w2},
\[
\W \left(\Law(x_k),\pi_\gamma\right)\le \W\left(\Law(\ov x_k),\pi_\gamma\right)+\frac{\eta\delta(1-q^k)}{1-q}.
\]
Applying the exact-chain bound \eqref{eq:exact_floor_w2_assumption} ensures
\[
\mathcal W_2\left(\Law(x_k),\pi_\gamma\right)\le \rho^k \mathcal W_2\left(\Law(x_0),\pi_\gamma\right)+B_{\mathrm{disc}}(\eta)+\frac{\eta\delta(1-q^k)}{1-q},
\]
which consequently leads to \eqref{eq:imula_complete_w2}. Letting $k\to\infty$ implies that the inequality \eqref{eq:imula_stationary_floor} holds true, which completes the proof.
\end{proof}

\section{Experiments}
\label{sec:experiments}

We evaluate iPULA on a nonsmooth Gaussian image deblurring problem with total variation regularization. Given a ground-truth image $x_{\mathrm{true}} \in \mathbb{R}^{n\times n}$, the observation model is
\begin{equation}
y = Hx_{\mathrm{true}} + \varepsilon,
\qquad
\varepsilon \sim \mathcal{N}(0,\sigma^2 I),
\end{equation}
where $H$ is a spatially invariant blur operator and $\varepsilon$ is additive Gaussian noise. We sample from the posterior associated with the composite potential
\begin{equation}
U(x)
=
\frac{1}{2\sigma^2}\|Hx-y\|^2
+
\lambda_{\mathrm{TV}}\mathrm{TV}(x)
+
\frac{\lambda_2}{2}\|x\|^2,
\label{eq:brain_potential}
\end{equation}
where $\mathrm{TV}(x)$ denotes anisotropic total variation. This setting is representative of strongly convex nonsmooth imaging inverse problems for which the proximal operator of the full composite potential is not available in closed form.

Experiments are conducted on a BrainWeb T1-weighted MRI slice resized to $128\times128$. The BrainWeb data set is publicly available\footnote{\url{https://brainweb.bic.mni.mcgill.ca/brainweb/}}. The degradation operator is a spatially invariant $5\times5$ box blur implemented by FFT convolution with periodic boundary conditions. Gaussian noise is added to obtain a blurred signal-to-noise ratio (BSNR) of $40$ dB. All methods are initialized from the same Wiener deconvolution estimate. All experiments were implemented in Python 3.11 on a MacBook Pro with Apple M1 Pro and 32 GB RAM. 
Code is available at: \url{https://github.com/susanGhaderi/iPULA-BrainWeb}.

We compare iPULA against three Langevin-type baselines: (i) \emph{Grad-sub} from \cite{habring2024subgradient}; (ii) \emph{Prox-sub} from \cite{habring2024subgradient}; and (iii) MYULA from \cite{DurmusPereyra}. We use
\[
\lambda_{\mathrm{TV}} = 10^{-3},
\qquad
\lambda_2 = 10^{-2}.
\]
All methods are run for $5000$ Langevin iterations. For all the methods, we use the Moreau parameter $\gamma=10^{-6}$ and for iPULA and MYULA, we use Langevin step-size $\eta=0.4\times10^{-6}$. Reconstruction quality is measured using PSNR and SSIM, both computed with respect to the ground-truth image. Following common Bayesian imaging practice, PSNR and SSIM have been reported to provide the best MAP-like reconstruction along the trajectory.

\begin{figure}[hbpt!]
    \centering

    \begin{subfigure}[t]{0.36\linewidth}
        \centering
        \includegraphics[width=\linewidth]{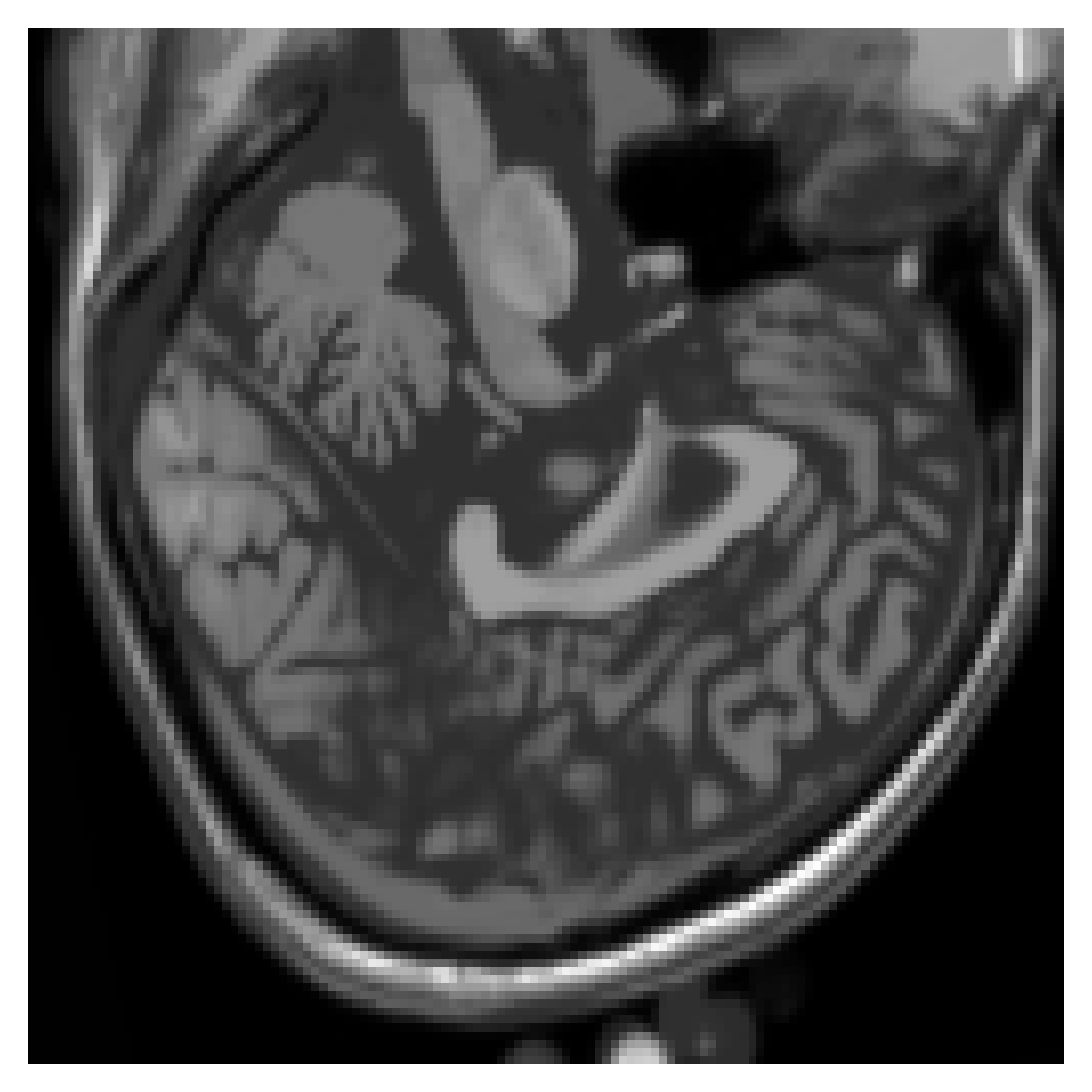}
        \caption{Ground truth.}
    \end{subfigure}\qquad\qquad
    \begin{subfigure}[t]{0.36\linewidth}
        \centering
        \includegraphics[width=\linewidth]{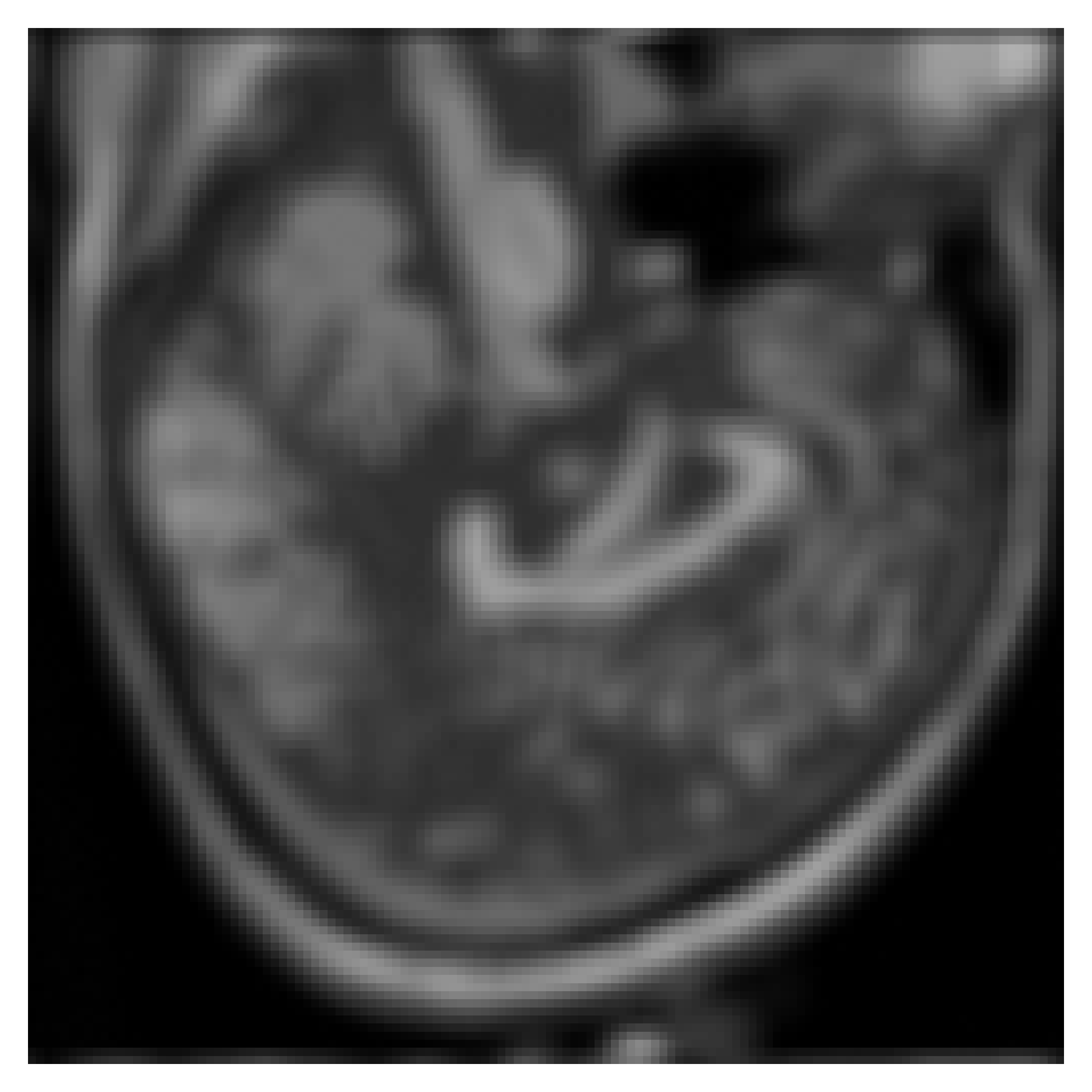}
        \caption{Blurred-noisy image.}
    \end{subfigure}\qquad\qquad
    \begin{subfigure}[t]{0.36\linewidth}
        \centering
        \includegraphics[width=\linewidth]{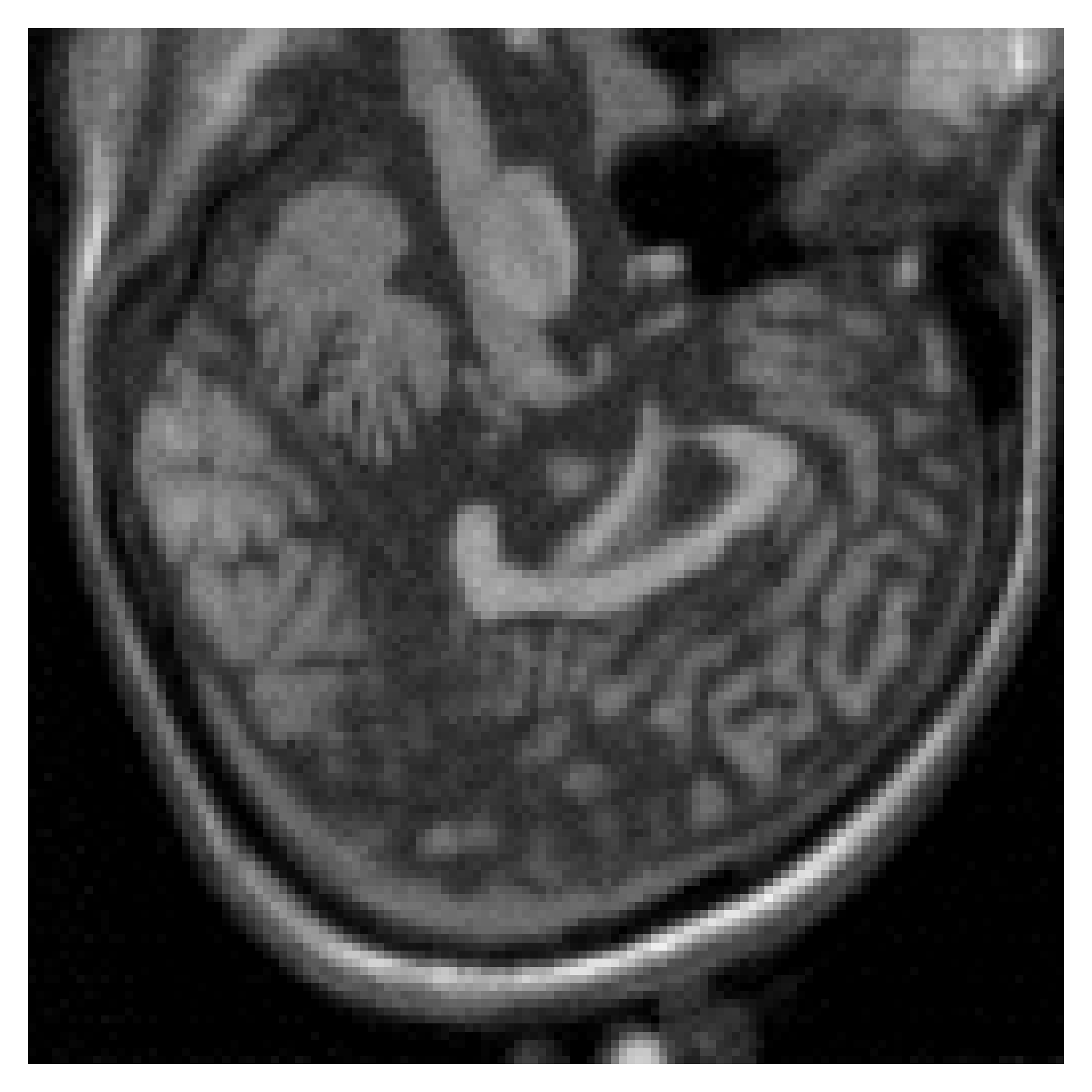}
        \caption{MYULA ($31.65$ dB).}
    \end{subfigure}\qquad\qquad
    \begin{subfigure}[t]{0.36\linewidth}
        \centering
        \includegraphics[width=\linewidth]{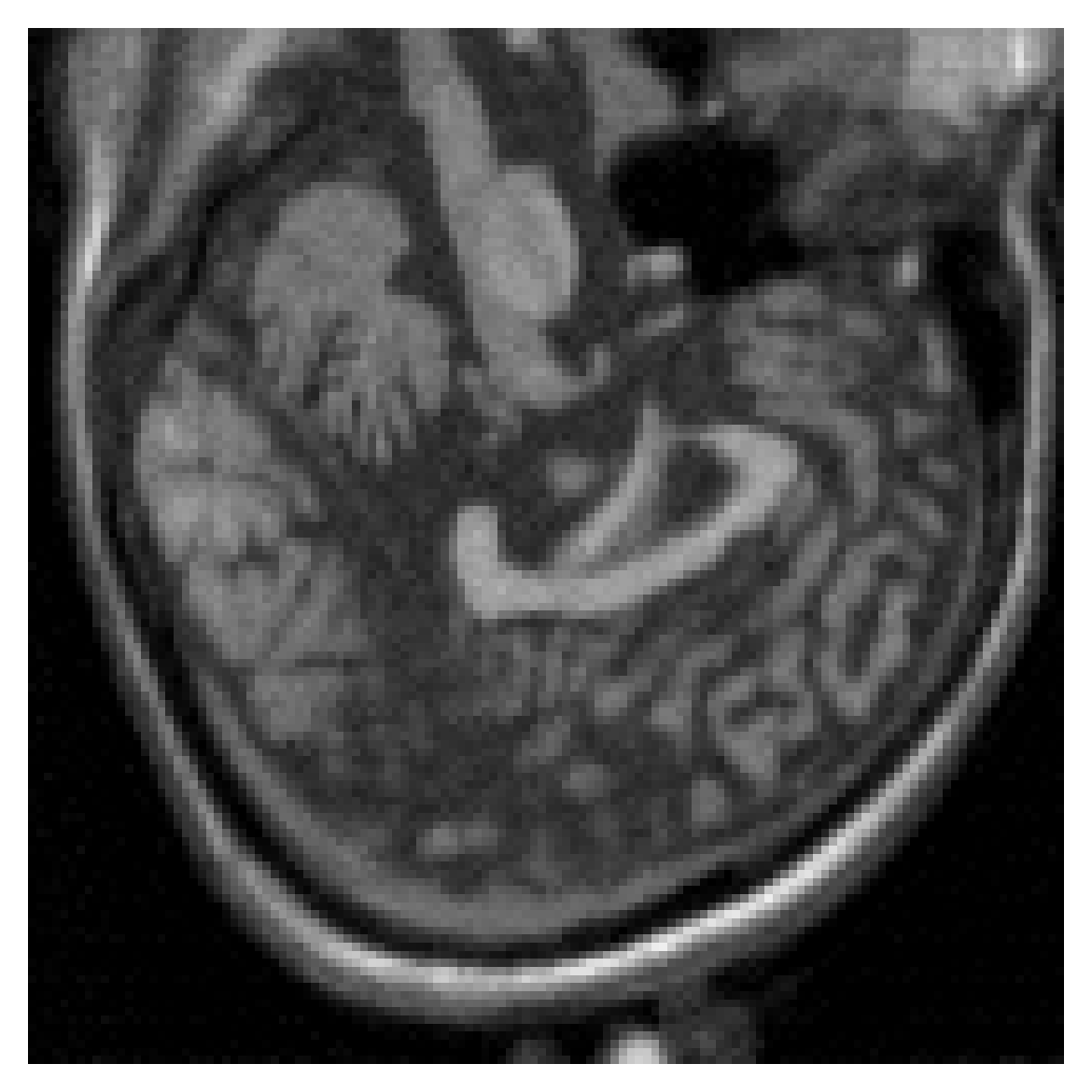}
        \caption{iPULA ($31.61$ dB).}
    \end{subfigure}\qquad\qquad
    \begin{subfigure}[t]{0.36\linewidth}
        \centering
        \includegraphics[width=\linewidth]{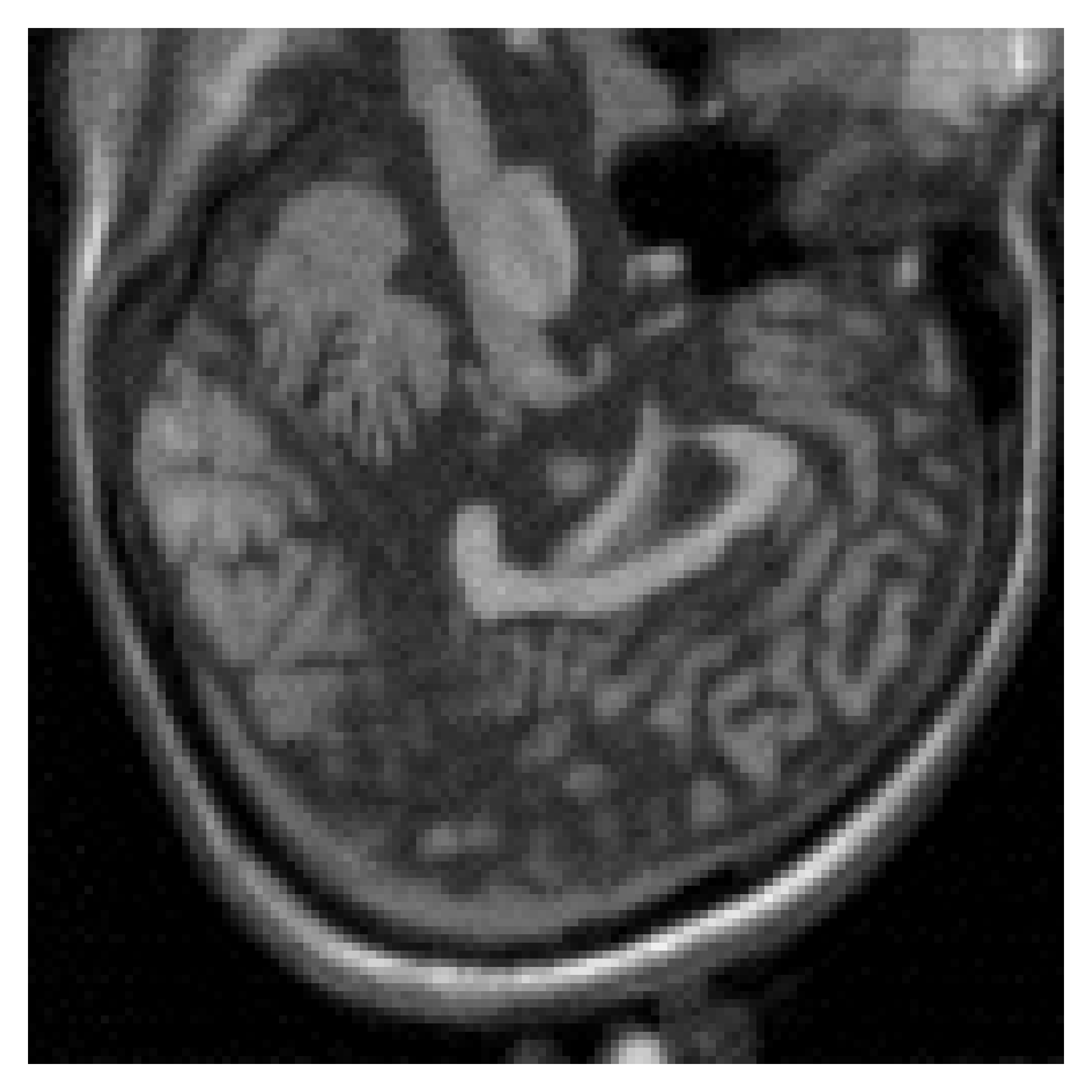}
        \caption{Grad-sub ($31.57$ dB).}
    \end{subfigure}\qquad\qquad
    \begin{subfigure}[t]{0.36\linewidth}
        \centering
        \includegraphics[width=\linewidth]{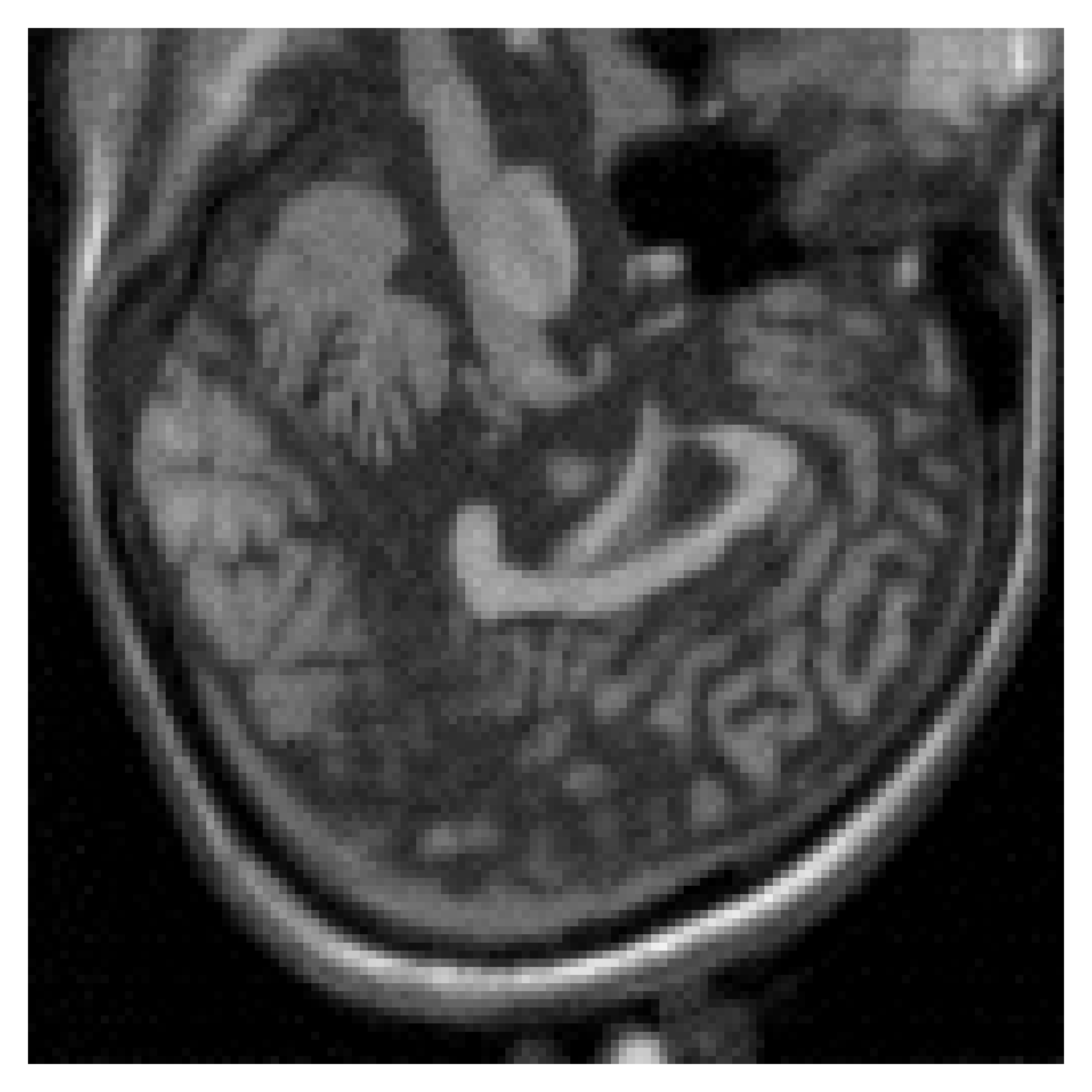}
        \caption{Prox-sub ($31.53$ dB).}
    \end{subfigure}

    \caption{
    Brain MRI deblurring results for MYULA, iPULA, Grad-sub, and Prox-sub.
    }
    \label{fig:brain_recons}
\end{figure}

Figure~\ref{fig:brain_recons} shows that all methods substantially improve over the blurred-noisy observation and recover the main anatomical structures of the MRI slice. The final reconstruction quality is highly competitive across methods: iPULA reaches $31.61$ dB, essentially matching Grad-sub ($31.57$ dB), Prox-sub ($31.53$ dB), and MYULA ($31.65$ dB). This confirms that the proposed inexact full-proximal Langevin update can achieve the same reconstruction quality as splitting-based Langevin baselines, despite using only an approximate solution of the full proximal subproblem.
\begin{figure}[hbpt!]
    \centering

    \begin{subfigure}[t]{0.38\linewidth}
        \includegraphics[width=\linewidth]{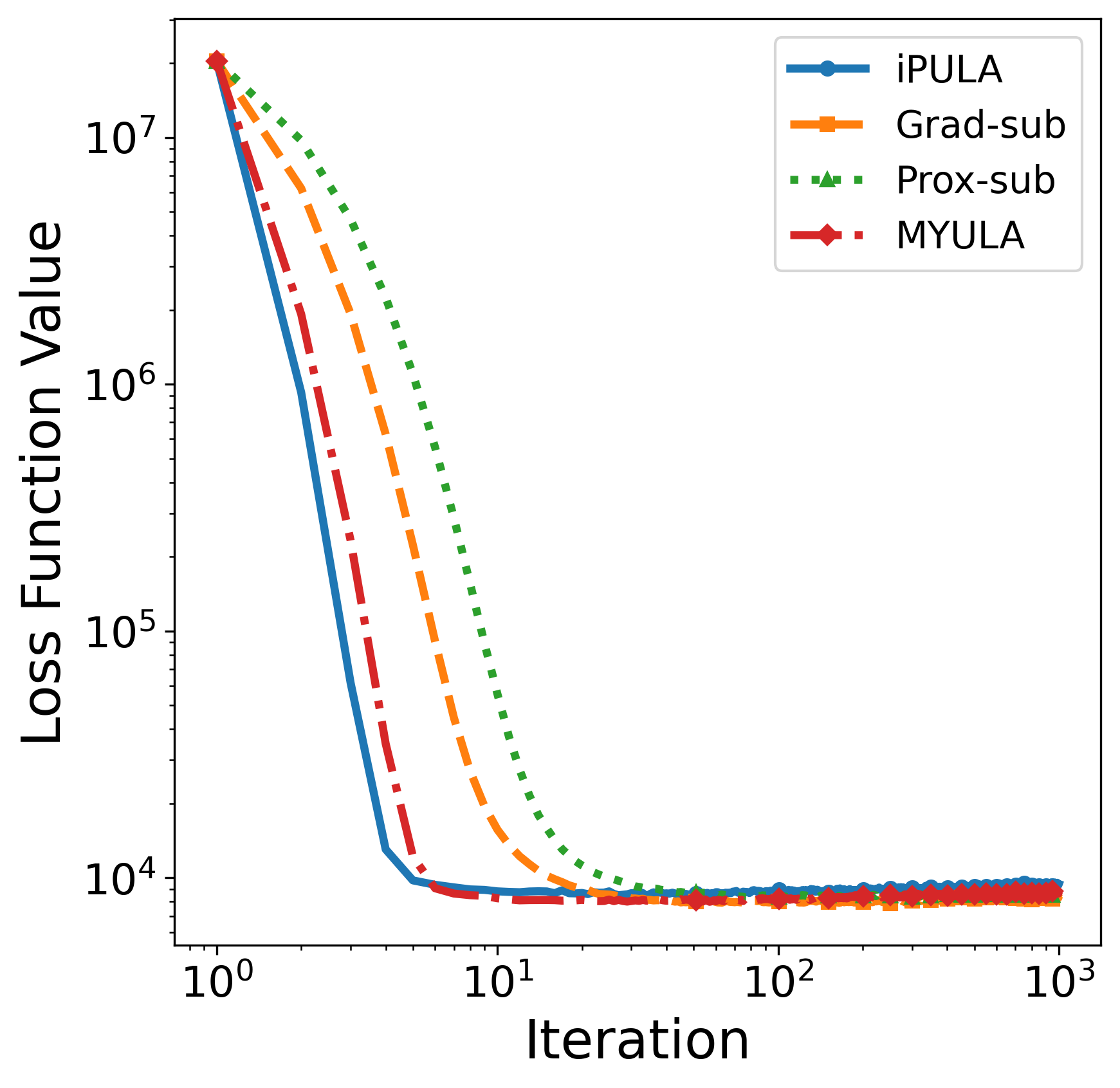}
        \caption{Objective value.}
    \end{subfigure}\qquad\qquad
    \begin{subfigure}[t]{0.4\linewidth}
        \includegraphics[width=\linewidth]{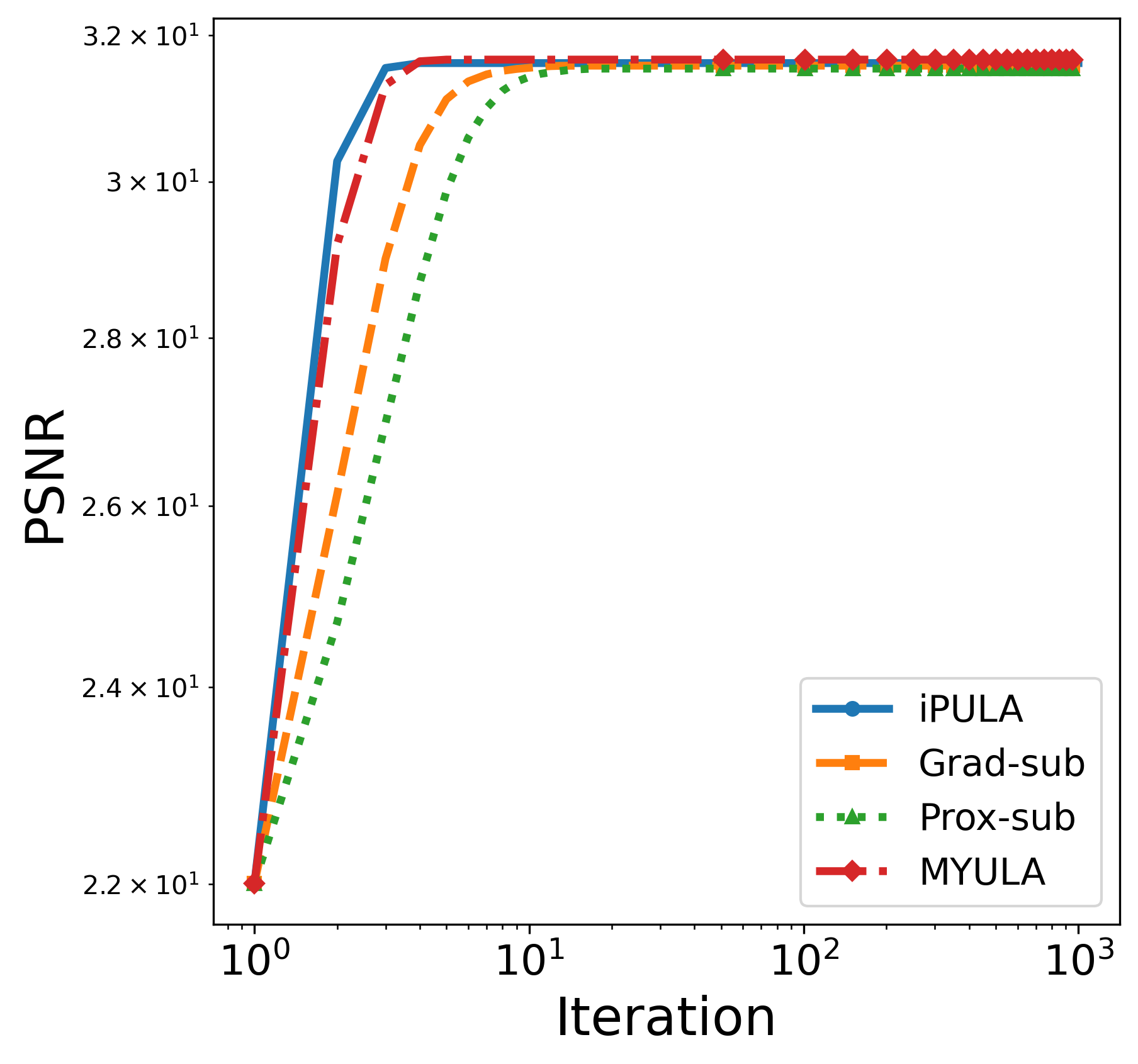}
        \caption{PSNR.}
    \end{subfigure}\qquad\qquad
    \begin{subfigure}[t]{0.4\linewidth}
        \includegraphics[width=\linewidth]{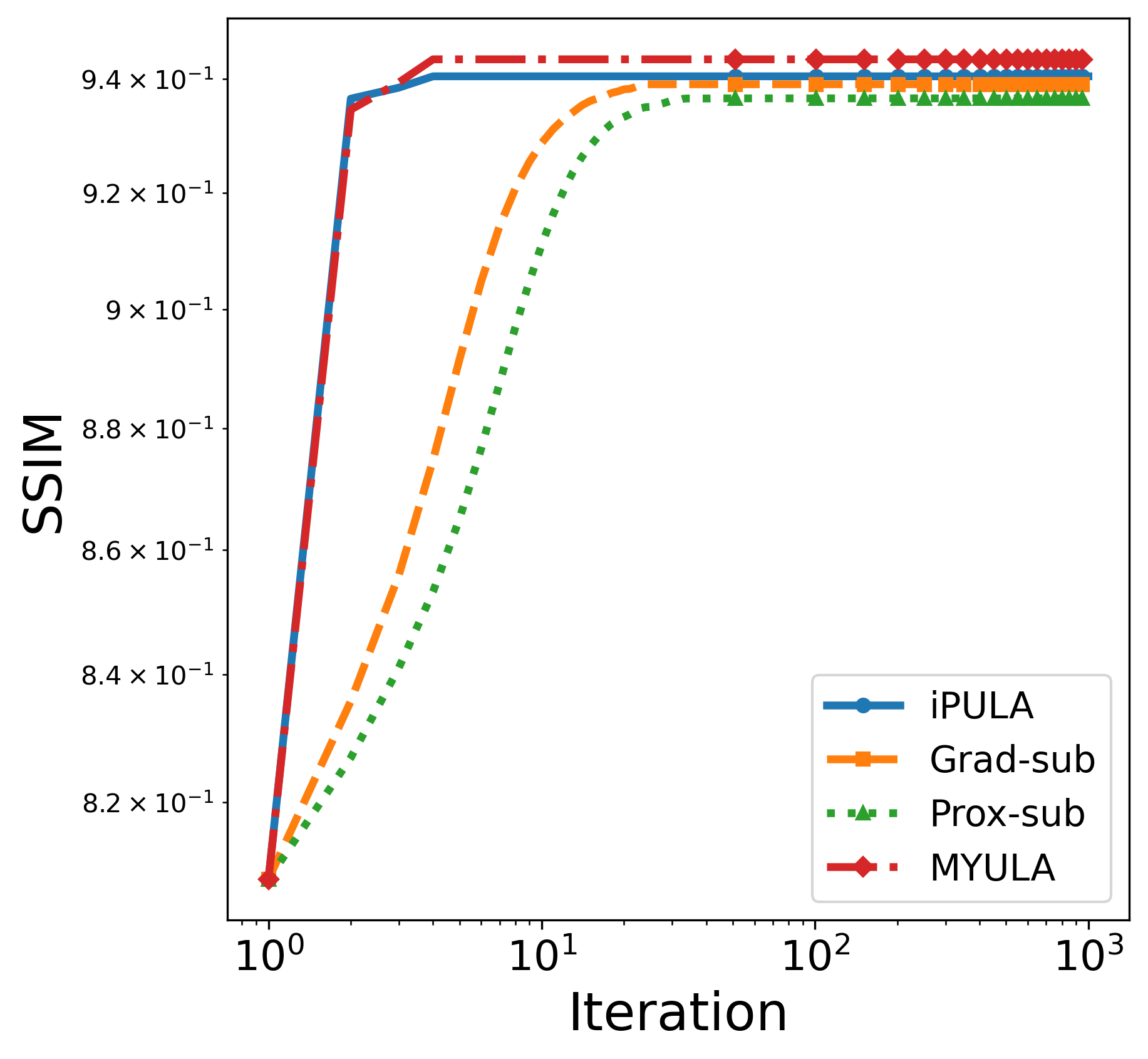}
        \caption{SSIM.}
    \end{subfigure}\qquad\qquad
    \begin{subfigure}[t]{0.38\linewidth}
        \centering
        \includegraphics[width=\linewidth]{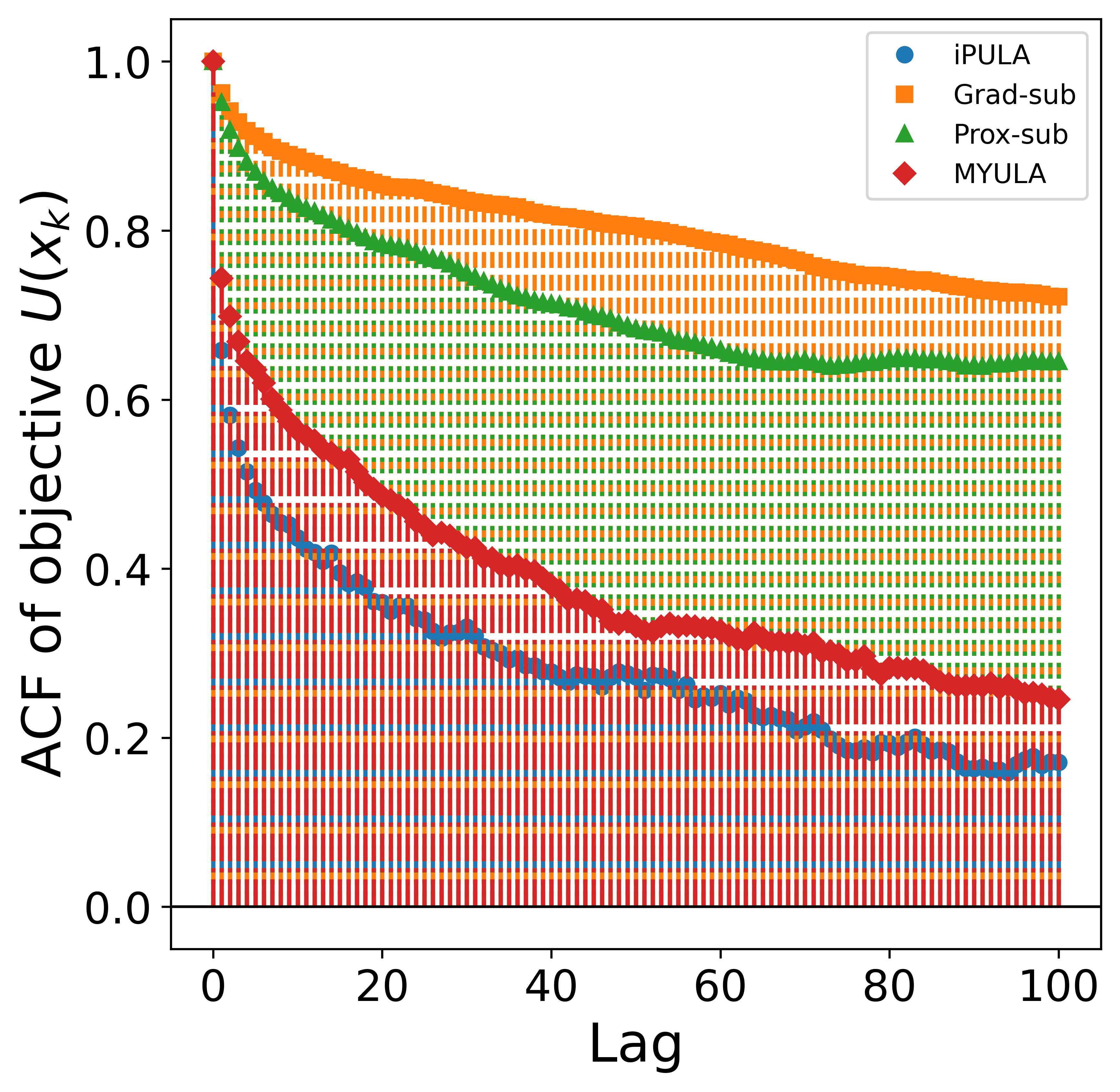}
        \caption{ACF.}
    \end{subfigure}

    \caption{
    Convergence and sampling behavior of iPULA, Grad-sub, Prox-sub, and MYULA on the BrainWeb deblurring problem.
    From left to right: objective value, PSNR, and  SSIM versus iteration number (log scale), and autocorrelation function (ACF) of the objective sequence.
    }
    \label{fig:brain_curves}
\end{figure}

Figure~\ref{fig:brain_curves} highlights the transient behavior of the methods. iPULA rapidly reaches the high-quality reconstruction regime in both PSNR and SSIM, with behavior comparable to MYULA and faster than Grad-sub and Prox-sub. The objective curve also shows that iPULA rapidly decreases the composite potential and stabilizes near the same objective range as competing methods. This is important because iPULA does not require an exact proximal operator: the full proximal step is computed only approximately by an inner projected subgradient procedure; see, e.g., \cite{mordukhovich2014easy}.

To assess sampling efficiency, Subfigure~\ref{fig:brain_curves}~(d) reports the autocorrelation function (ACF) of the post-burn-in objective sequence. The ACF of iPULA decays substantially faster than those of Grad-sub, Prox-sub, and MYULA, indicating a lower correlation between successive samples of the posterior energy. This provides evidence that the inexact full-proximal Moreau-gradient direction enhances posterior exploration in addition to improving reconstruction quality.

Overall, the results support the main premise of this work: proximal Langevin sampling with inexact proximal evaluations can remain stable and accurate even when the proximal operator of the full nonsmooth composite potential is unavailable in closed form. In this BrainWeb deblurring experiment, iPULA achieves a reconstruction quality comparable to strong splitting-based baselines while exhibiting favorable objective-level mixing behavior.

\section{Discussion}\label{sec:discussion}
In this work, we studied sampling from posterior distributions with nonsmooth composite potentials, where the proximal operator of the nonsmooth component was not available in closed form—a setting that arose frequently in Bayesian learning. To address this challenge, we computed the proximal operator approximately and constructed an inexact oracle for the Moreau envelope. Building on this, we proposed iPULA, an inexact proximal unadjusted Langevin algorithm that leveraged approximate gradients of the Moreau envelope, enabling efficient sampling in problems with complex nonsmooth structures. 
On the theoretical side, we showed that, under appropriately chosen inexactness criteria for the proximal subproblem, iPULA enjoyed non-asymptotic convergence guarantees. Empirically, experiments on medical image recovery with total variation regularization demonstrated the strong practical performance of iPULA, corroborating our theoretical findings.

An appealing feature of our framework is its flexibility with respect to the structure of the potentials, as it only requires the associated proximal subproblem to be solved efficiently. In our convex setting, this subproblem could be addressed using a variety of optimization techniques, including subgradient methods (see \cite{davis2018subgradient,mordukhovich2014easy,rahimi2024projected}); however, to the best of our knowledge, it remained unclear which approach was the most efficient for solving this subproblem in practice.

\section*{Acknowledgments}\label{sec:acknowledgnebts}
 The first author acknowledges the support of the Research Foundation Flanders (FWO) research project 12AK924N. MA was partially supported by the Research Foundation Flanders (FWO) research project G081222N and UA BOF DocPRO4 projects with ID 46929 and 48996.

\bibliographystyle{plain}

\end{document}